\documentclass[preprint,12pt]{elsarticle}

\usepackage{amsthm}
\usepackage{amsmath,amssymb,txfonts}
\usepackage{color,bm,comment}
\usepackage{comment}
\usepackage[english]{babel}

\newtheorem{theorem}{Theorem}
\newtheorem{lemma}{Lemma}
\newtheorem{prop}{Proposition}

\newtheorem{remark}{Remark}

\theoremstyle{definition}

\def\re{\mathbb{R}}
\def\N{\mathbb{N}}

\def\({\left(}
\def\){\right)}
\def\[{\left[}
\def\]{\right]}
\def\pd{\partial}

\def\ep{\varepsilon}
\def\w{\omega}
\def\la{\lambda}

\def\tcr{\textcolor{red}}

\newtheorem{ThmA}{Theorem A}

\newtheorem{ThmB}{Theorem B}

\newtheorem{ThmD}{Theorem D}

\begin{document}

\begin{frontmatter}



\title{Critical Caffarelli-Kohn-Nirenberg inequalities\\ via nonlinear scaling argument}


\author[MS]{Megumi Sano\corref{Sano}\fnref{label1}}
\ead{sano@cc.nara-wu.ac.jp}
\fntext[label1]{Corresponding author.}

\author[YSa,YSb]{Yerkin Shaimerdenov}
\ead{yerkin.shaimerdenov@sdu.edu.kz}

\author[HS]{Hinako Sugimura}
\ead{sugimura.hinako@gmail.com}

\address[MS]{Department of Mathematics, Nara Women's University, 
Nara, 630-8506, Japan}
\address[YSa]{SDU University, 1/1 Abylai Khan Str., Kaskelen, Kazakhstan}
\address[YSb]{Institute of mathematics and mathematical modeling, 125 Pushkin Str., Almaty, Kazakhstan}
\address[HS]{Graduate School of Humanities and Sciences, Nara Women's University, Nara, 630-8506, Japan}


\date{\today}

\begin{abstract}
We show functional inequalities with logarithmic weights as a critical case of Caffarelli-Kohn-Nirenberg inequalities. Original CKN inequalities have polynomial weights and are shown via scaling argument. In this paper, we develop this argument into it for nonlinear scaling.  
Our inequalities are generalizations of critical Hardy type inequalities. 
\end{abstract}

\end{frontmatter}



%
%
\section{Introduction and main theorem}\label{S Intro}

Let $n \in \N$. In 1984, Caffarelli, Kohn and Nirenberg showed the following inequalities with polynomial weights. 
We give an equivalent form of their result.

\begin{ThmA}(\cite{CKN})
Let $0< a \le 1, p, q \ge 1, r>0$ and $\alpha, \beta, \gamma > -n$. 
Then there exists $C>0$ such that for any functions $u \in C_0^\infty (\re^n)$ the following inequality holds:
\begin{align}\label{CKN ineq}
    C \left( \int_{\re^n} |x|^\gamma |u(x)|^r \,dx \right)^{\frac{1}{r}}
    \le \left( \int_{\re^n} |x|^\alpha |\nabla u(x)|^p \,dx \right)^{\frac{a}{p}} \left( \int_{\re^n} |x|^\beta |u(x)|^q \,dx \right)^{\frac{1-a}{q}}
\end{align}
if and only if the following conditions hold.
\begin{align}\label{dilation valance}
    &\text{(Dilation balance at $0$)}\quad 
    \frac{\gamma +n}{r} = \frac{\alpha +n -p}{p} a \,+ \frac{\beta + n}{q}(1-a),\\
    \label{translation valance}
    &\text{(Translation balance)}\quad 
    \frac{1}{r}\ge \frac{a}{p} + \frac{1-a}{q} -\frac{a}{n},\\
    \label{valance at 0}
    &\text{(Balance at $0$)} \quad
    \frac{1}{r} \le \frac{a}{p}  + \frac{1-a}{q}
    \,\,\,\text{if} \,\,\, \frac{\alpha +n -p}{p}= \frac{\gamma + n}{r}
\end{align}
\end{ThmA}

\begin{remark}\label{Rem a=1}
    If $a<1$, $\frac{\alpha +n -p}{p}= \frac{\gamma + n}{r}$
     implies $\frac{\alpha +n -p}{p}= \frac{\gamma + n}{r}=\frac{\beta+n}{q}$ thanks to (\ref{dilation valance}). 
\end{remark}

The inequalities (\ref{CKN ineq}) include many inequalities, for example, the Hardy, the Sobolev and the H\'enon inequalities ($a=1$), the Nash inequality ($\alpha =\beta = \gamma =0, \,p=r=2, q=1$), the Gagliardo-Nirenberg inequality ($\alpha = \beta = \gamma =0$) and the Heisenberg uncertainty principle ($\alpha = \gamma =0, \,\beta =p=q=r=2$) and so on. There are various applications of these inequalities to PDEs, see e.g. \cite{Nash,CKN(PDE),BG,BV,SSW}, to name a few.  

In Theorem A, we assume $\alpha, \beta, \gamma >-n$ because the weights $|x|^\alpha, |x|^\beta, |x|^\gamma$ should belong to $L^1_{\rm loc}(\re^n)$. 
If we change the functional space $C_0^\infty (\re^n)$ to a smaller space $C_0^\infty(\re^n \setminus \{ 0\})$ in Theorem A, we can expect to drop these assumptions with respect to $\alpha, \beta, \gamma$. However, even if we consider $C_0^\infty(\re^n \setminus \{ 0\})$, a case where 
\begin{align}\label{critical case}
     (\alpha, \,\beta, \,\gamma ) = (p-n, \,-n, \,-n)
\end{align}
is still excluded as follows. Due to this, we call (\ref{critical case}) {\it the critical case} of CKN inequalities (\ref{CKN ineq}) in this paper. We will show Proposition \ref{Prop critical case} in \S \ref{S Prop}. 

\begin{prop}\label{Prop critical case}
    Let $n \ge 2$
    , $0< a \le 1, p, q \ge 1, r>0$ satisfy (\ref{translation valance}) in Theorem A. 
 Then
 \begin{align}\label{critical dame}
     \inf_{0 \not\equiv u \in C_0^\infty (\re^n \setminus \{ 0\})} \frac{\left( \int_{\re^n} |x|^{p-n} |\nabla u(x)|^p \,dx \right)^{\frac{a}{p}}\left( \int_{\re^n} |x|^{-n} |u(x)|^q \,dx \right)^{\frac{1-a}{q}}}{\left( \int_{\re^n} |x|^{-n} |u(x)|^r \,dx \right)^{\frac{1}{r}}} =0.
 \end{align}
\end{prop}


\noindent
The above fact (\ref{critical dame}) suggests that we consider another weights in the critical case (\ref{critical case}). 
In addition, even in the case where $a=1$, it is known that for any nonnegative weights $g \in L^1_{\rm loc}(\re^n)$ such that $\int_{B_1} g(x) \,dx >0$ and for any $r >0$ and $p>1$, 
we have
\begin{align}\label{whole space dame}
     \inf_{0 \not\equiv u \in C_0^\infty (\re^n )} \frac{\left( \int_{\re^n} |x|^{p-n} |\nabla u(x)|^p \,dx \right)^{\frac{1}{p}}}{\left( \int_{\re^n} g(x) \,|u(x)|^r \,dx \right)^{\frac{1}{r}}} =0,
 \end{align}
 see e.g. \cite[Remark 1.2 (2)]{HK} or \S \ref{S Prop}. The above fact (\ref{whole space dame}) suggests that we consider bounded domains, especially balls $B_R=\{ x \in \re^n\,\,|\,\, |x| <R\}$ in the critical case (\ref{critical case}). 
 Our main result is the critical case of CKN inequalities (\ref{CKN ineq}) as follows.

\begin{theorem}\label{Thm CCKN}
Let $0< a \le 1, p, q \ge 1, r>0, \alpha \in \re$ and $\beta, \gamma >1$. 
Then there exists $C>0$ such that for any functions $u \in C_0^\infty (B_R)$ the following inequality holds:
\begin{align}\label{CCKN ineq}
C \left( \int_{B_R} \frac{|u(x)|^r}{|x|^n (\log \frac{eR}{|x|})^\gamma} \,dx \right)^{\frac{1}{r}} \le \left( \int_{B_R} \frac{|x|^p \,|\nabla u(x)|^p}{|x|^n (\log \frac{eR}{|x|})^\alpha} \,dx \right)^{\frac{a}{p}} \left( \int_{B_R} \frac{|u(x)|^q}{|x|^n (\log \frac{eR}{|x|})^\beta} \,dx \right)^{\frac{1-a}{q}}
\end{align}
if and only if the following conditions hold.
\begin{align}\label{n-dilation valance}
    &\text{(Nonlinear scale balance at $0$)}\quad \frac{\gamma -1}{r} \ge \frac{\alpha -1 +p}{p} a \,+ \frac{\beta-1}{q}(1-a),\\
    \label{n-valance at mid}
    &\text{(Dilation balance at a midpoint)}\quad 
    \frac{1}{r}\ge \frac{a}{p} + \frac{1-a}{q} -\frac{a}{n},\\
    \label{n-valance at 0}
    &\text{(Balance at $0$)} \quad \frac{1}{r} \le \frac{a}{p}  + \frac{1-a}{q}\,\,\,\text{if} \,\,\, \frac{\alpha -1+p}{p}= \frac{\gamma -1}{r}=\frac{\beta -1}{q}
\end{align}
\end{theorem}

\begin{remark}
     If $a=1$, we disregard $\beta$ and $q$ in Theorem \ref{Thm CCKN}. 
     Namely, we regard \eqref{n-dilation valance}-\eqref{n-valance at 0} as 
    \begin{align*}
        \frac{\gamma -1}{r} \ge \frac{\alpha -1 +p}{p},\,\,
        \frac{1}{r} \ge \frac{1}{p} -\frac{1}{n},\,\,{\rm and}\,\,
        \frac{1}{r} \le \frac{1}{p}\,\,\,{\rm if}\,\,\, \frac{\alpha -1 +p}{p}=\frac{\gamma -1}{r}.
    \end{align*}
    Similarly, we disregard $\beta$ (or $\beta'$) and $q$ in Proposition \ref{Prop CKN ball}, Theorem \ref{Thm CCKN n=1} and Lemma \ref{lem CKN ball n=1}. 
\end{remark}

\begin{remark}(Another critical CKN inequalities)
Our inequality \eqref{CCKN ineq} is a generalization of the following critical Hardy inequality firstly founded by Leray\cite{Leray}:
\begin{align*}
    C \int_{B_R} \frac{|u(x)|^n}{|x|^n(\log \frac{eR}{|x|})^n} \,dx \le \int_{B_R} |\nabla u(x)|^n\,dx\quad (u \in C_0^\infty (B_R),\, n \ge 2),
\end{align*}
while there is also a generalization of another critical Hardy inequality:
\begin{align*}
    C \int_{\re^n} |x|^{-n} |u(x)|^n \,dx \le \int_{\re^n} |\log |x|\, |^n \,|\nabla u(x)|^n\,dx \quad (u \in C_0^\infty (\re^n \setminus \{ 0\})),
\end{align*}
see e.g. \cite{RSY,RSY2}, to name a few. 
\end{remark}

\noindent
Note that the restrictions $\beta, \gamma >1$ in Theorem \ref{Thm CCKN} come from $L^1$-integrability of the logarithmic weights. 
The inequalities (\ref{CCKN ineq}) with $a=1$ become critical Hardy type inequalities with polynomial weights and logarithmic weights, which are studied in many papers \cite{Leray,AS,HK,MOW(Tohoku),II,S(JDE),H,P,S(MJM),Sano-TF(JGA)}. However, to the best of our knowledge, the inequalities (\ref{CCKN ineq}) with $a<1$ are new. 



\noindent
{\bf Outline of the paper}: The remainder of this paper is organized as follows. 

\noindent
{\it Section \ref{S Pre}}: In \S \ref{S Prop}, 
we will show Proposition \ref{Prop critical case} and (\ref{whole space dame}). Furthermore, we will derive critical CKN inequalities (\ref{CCKN ineq}) from original CKN inequalities (\ref{CKN ineq}) via a limiting procedure. These facts give us a reason why we consider logarithmic weights in the critical case. 
In \S \ref{S ball}, we will show Proposition \ref{Prop CKN ball}, which is concerning to CKN inequalities (\ref{CKN ineq}) on balls, which are used in \S \ref{S n=1}. In \S \ref{S nonlinear}, we will explain nonlinear scaling which is a key tool of this paper. Note that the critical CKN inequalities (\ref{CCKN ineq}) do not have scale invariance in general. The lack of scale invariance makes problems more difficult. For the details, see \S \ref{S nonlinear} and Remark \ref{Rem scale} in \S \ref{S a<1 n>1 last}. 

\noindent
{\it Section \ref{S Necessity}}: By using several test functions, we will show that if critical CKN inequalities (\ref{CCKN ineq}) hold, then the conditions (\ref{n-dilation valance})-(\ref{n-valance at 0}) in Theorem \ref{Thm CCKN} hold. 
From \S \ref{S a=1} to \S \ref{S a<1}, we will show the converse, that is, to show (\ref{CCKN ineq}) by using (\ref{n-dilation valance})-(\ref{n-valance at 0}). 

\noindent
{\it Section \ref{S a=1}}: 
We will show critical CKN inequalities (\ref{CCKN ineq}) with $a=1$ by using scaling argument by \cite{CKN}. Although the case where $a=1$ is already known by using nonlinear potential theory and generalized rearrangement technique, we give a proof because it is self-contained and is based on nonlinear scaling which is a simpler tool. For the details, see \S \ref{S a=1}. 

\noindent
{\it Section \ref{S n=1}}: We will show critical CKN inequalities (\ref{CCKN ineq}) with $n=1$ by using some transformation and Proposition \ref{Prop CKN ball}.

\noindent
{\it Section \ref{S a<1}}: We will complete the proof of Theorem \ref{Thm CCKN}. We will divide two cases. First, in \S \ref{S a<1 n>1}, we will consider the case where $a<1, n \ge 2$ and $\frac{1}{r} \le \frac{a}{p}+ \frac{1-a}{q}$. The strategy of the proof is the same as it in \S \ref{S a=1}. Next, in \S \ref{S a<1 n>1 last}, we will consider the case where $a<1, n \ge 2$ and $\frac{1}{r} > \frac{a}{p}+ \frac{1-a}{q}$. In this case, we will reduce it to \S \ref{S a<1 n>1}.

\noindent
{\it Section \ref{S App}}: For reader's convenience, we will give outline of a proof of Gagliardo-Nirenberg inequalities with mean zero, which are used in \S \ref{S a<1 n>1}.

\noindent
{\bf Notations}: Throughout this paper, if $u$ is a radial function that should be written as $u(x) = \tilde{u}(|x|)$ by some function $\tilde{u} = \tilde{u}(r)$, we write $u(x)= u(|x|)$ with admitting some ambiguity. Set 
$X_{\rm rad} = \{ u \in X \,\,|\,\, u(x) = u(|x|) \,\}$. 
Also, we use $C$ as a positive constant whose exact value is immaterial and may change from line to line. We use $\w_n$ as the surface area of the unit sphere $\mathbb{S}^{n-1}$ in $\re^n$. For unified notation, we define $\w_1 =2$.

%
%

\section{Preliminaries}\label{S Pre}

%
%

\subsection{Why do we consider logarithmic weights and balls in the critical case?}\label{S Prop}

First, we show Proposition \ref{Prop critical case} which gives us a reason why we consider the logarithmic weighted inequalities (\ref{CCKN ineq}) in the critical case (\ref{critical case}).  

\begin{proof}[Proof of Proposition \ref{Prop critical case}] 
We consider the following test function $\phi_\ep \in C_0(\re^n \setminus \{ 0\})$ for $\ep \in (0, 1)$.
\begin{align*}
    \phi_\ep (|x|) = \begin{cases}
        \log \frac{|x|}{\ep} \quad &\text{if} \quad  \ep \le |x| \le 1,\\
        \log \frac{1}{\ep |x|} \quad &\text{if} \quad  1 < |x| < \frac{1}{\ep},\\
        0 &\text{if}\quad  |x| < \ep \,\,{\rm or}\,\, |x| \ge \frac{1}{\ep}
    \end{cases}
\end{align*}
The function $\phi_\ep$ is not in $C_0^\infty (\re^n \setminus \{ 0\})$, but its mollification $\eta_\delta * \phi_\ep$ is in $C_0^\infty (\re^n \setminus \{ 0\})$ for small $\delta >0$, where 
\begin{align*}
\eta_\delta (|x|) =\delta^{-n} \,\eta \,(|x|/\delta) =
    \begin{cases}
        C \delta^{-n} \exp \( \frac{\delta^2}{|x|^2 -\delta^2} \) \quad &\text{if}\,\, |x| < \delta,\\
        0&\text{if}\,\, |x| \ge \delta. 
    \end{cases}
\end{align*}
Therefore, testing by $\eta_\delta * \phi_\ep$ instead of $\phi_\ep$ and taking a limit of each integration as $\delta \to 0$, we can verify that there is no problem to consider $\phi_\ep$ as a test function. 
Then we have
\begin{align*}
    \int_{\re^n} |x|^{p-n}|\nabla \phi_\ep|^p \,dx 
    &=\w_n \left[ \int_{\ep}^1 r^{p-1} \left| \frac{1}{r} \right|^p\,dr +  \int_1^{\frac{1}{\ep}} r^{p-1} \left| -\frac{1}{r} \right|^p\,dr \right]
    =2\w_n \log \frac{1}{\ep},\\
    \int_{\re^n} |x|^{-n}|\phi_\ep|^q \,dx 
    &=\frac{2\w_n}{q+1} \left( \log \frac{1}{\ep} \right)^{q+1},  \,\, \int_{\re^n} |x|^{-n}|\phi_\ep|^r \,dx 
    =\frac{2\w_n}{r+1} \left( \log \frac{1}{\ep} \right)^{r+1}, 
\end{align*}
which imply that 
\begin{align*}
    \frac{\left( \int_{\re^n} |x|^{p-n} |\nabla \phi_\ep|^p \,dx \right)^{\frac{a}{p}}\left( \int_{\re^n} |x|^{-n} |\phi_\ep|^q \,dx \right)^{\frac{1-a}{q}}}{\left( \int_{\re^n} |x|^{-n} |\phi_\ep|^r \,dx \right)^{\frac{1}{r}}}
    &\le C \left( \log \frac{1}{\ep} \right)^{\frac{a}{p} + \frac{1-a}{q}-\frac{1}{r}-a}\\
    &\le C \left( \log \frac{1}{\ep} \right)^{-a(1-\frac{1}{n})} \to 0 \,\,(\ep \to 0).
\end{align*}
\end{proof}


Next, we show (\ref{whole space dame}) to give a reason why we cannot consider $\re^n$ in Theorem \ref{Thm CCKN}.

\begin{proof}[Proof of (\ref{whole space dame})] 
We consider the following test function $\phi_\ep \in C_0(\re^n)$ for $\ep \in (0, 1)$. 
\begin{align*}
    \phi_\ep (|x|) = \begin{cases}
        1 \quad &\text{if} \quad  0 \le |x| \le 1,\\
        \left( \log \frac{1}{\ep} \right)^{-1} \log \frac{1}{\ep |x|} \quad &\text{if} \quad  1 < |x| < \frac{1}{\ep},\\
        0 &\text{if}\quad  |x| \ge \frac{1}{\ep}
    \end{cases}
\end{align*}
Then we have
\begin{align*}
    \int_{\re^n} |x|^{p-n}|\nabla \phi_\ep|^p \,dx 
    =\w_n \left( \log \frac{1}{\ep} \right)^{1-p},\,\,
    \int_{\re^n} g(x)|\phi_\ep|^r \,dx 
    \ge \int_{B_1} g(x) \,dx >0 
\end{align*}
which imply that 
\begin{align*}
    \frac{\left( \int_{\re^n} |x|^{p-n} |\nabla \phi_\ep|^p \,dx \right)^{\frac{1}{p}}}{\left( \int_{\re^n} g(x) |\phi_\ep|^r \,dx \right)^{\frac{1}{r}}}
    \le C \left( \log \frac{1}{\ep} \right)^\frac{1-p}{p} \to 0 \,\,(\ep \to 0).
\end{align*}
\end{proof}

Finally, for radial functions, we shall derive our critical CKN inequalities (\ref{CCKN ineq}) with $\alpha >1-p$ by taking a limit of original CKN inequalities (\ref{CKN ineq}) via harmonic transplantation. Concerning to taking a limit of functional inequalities via harmonic transplantation, see \cite[Section 3.3]{ST(HT)}. This also gives us a reason why we consider the logarithmic weights in the critical case (\ref{critical case}). 
Set 
\begin{align*}
    C_{\rm rad}(\alpha, \beta, \gamma)= \inf_{0\not\equiv u \in C_{0, {\rm rad}}^\infty (\re^n \setminus \{ 0\})} \frac{\left( \int_{\re^n} |x|^{\alpha} |\nabla u|^p \,dx \right)^{\frac{a}{p}} \left( \int_{\re^n} |x|^\beta |u|^q \,dx \right)^{\frac{1-a}{q}}}{\left( \int_{\re^n} |x|^\gamma |u|^r \,dx \right)^{\frac{1}{r}}}
\end{align*}
as the best constant of (\ref{CKN ineq}) for radial functions when the conditions \eqref{dilation valance}-\eqref{valance at 0} are satisfied. 
We observe that even if we change the domain from $\re^n$ to $B_1$, the best constant does not change thanks to the scale invariance structure, see e.g. \cite[Section 4.5]{Struwe}. Also, we observe that 
\begin{align}\label{C order}
    &C_{\rm rad}(p-n+\ep c_1, -n+ \ep c_2, -n+ \ep c_3) \notag \\
    &=\ep^{\frac{1}{r}+ \frac{p-1}{p}a - \frac{1-a}{q}} \,C_{\rm rad}(p-n + c_1, -n+c_2, -n+c_3)
\end{align}
where $c_1, c_2, c_3 >0$ are constants and $\ep >0$ is a parameter. 
Indeed, if we use a nonlinear scaling $u_\ep (|x|) = u(|x|^\ep)$ for radial function $u$, we have
\begin{align*}
     &\frac{\left( \int_{\re^n} |x|^{p-n+\ep c_1} |\nabla u_\ep|^p \,dx \right)^{\frac{a}{p}} \left( \int_{\re^n} |x|^{-n+\ep c_2} |u_\ep|^q \,dx \right)^{\frac{1-a}{q}}}{\left( \int_{\re^n} |x|^{-n+\ep c_3} |u_\ep|^r \,dx \right)^{\frac{1}{r}}}\\
     &= \ep^{\frac{1}{r}+ \frac{p-1}{p}a - \frac{1-a}{q}} \,\, \frac{\left( \int_{\re^n} |x|^{p-n+c_1} |\nabla u|^p \,dx \right)^{\frac{a}{p}} \left( \int_{\re^n} |x|^{-n+c_2} |u|^q \,dx \right)^{\frac{1-a}{q}}}{\left( \int_{\re^n} |x|^{-n+c_3} |u|^r \,dx \right)^{\frac{1}{r}}}
\end{align*}
which implies \eqref{C order}. 

Let $v \in C_0^{\infty} (B_1)$ be a radial function, $(a ,p, q, r, \alpha, \beta, \gamma)$ satisfy
\begin{align}\label{n-valance 2}
    \frac{\gamma -1}{r} = \frac{\alpha-1 +p}{p}a + \frac{\beta -1}{q}(1-a)
\end{align}
and \eqref{n-valance at mid}-\eqref{n-valance at 0} and $\alpha >1-p$. For parameter a $\ep >0$, we set
\begin{align*}
    \alpha_\ep = p-n + \ep \, (\alpha -1 + p), \,\,\beta_\ep = -n+\ep \,(\beta -1), \,\,\gamma_\ep = -n + \ep \, (\gamma -1)
\end{align*}
and 
\begin{align*}
    u(|x|) = v(|y|), \,\text{where}\,\, e^\ep |y|^{-\ep} -1 = (e^\ep -1) |x|^{-\ep}. 
\end{align*}
Thanks to \eqref{n-valance 2} and \eqref{n-valance at mid}-\eqref{n-valance at 0}, we see that $(a, p, q, r, \alpha_\ep, \beta_\ep, \gamma_\ep)$ satisfies \eqref{dilation valance}-\eqref{valance at 0} and $(\alpha_\ep, \beta_\ep, \gamma_\ep)$ goes to $(p-n, -n, -n)$ which is the critical case (\ref{critical case}) as $\ep \to 0$. 
Since 
$$1-\left( \frac{|y|}{e} \right)^\ep = \ep \log \frac{e}{|y|} + o(\ep) \quad (\ep \to 0),$$ we have
\begin{align*}
    0&<C_{\rm rad}(p-n + \alpha+p -1, -n+ \beta -1, -n + \gamma -1) 
    = \ep^{-\frac{1}{r} - \frac{p-1}{p}a + \frac{1-a}{q}} \,C_{\rm rad}(\alpha_\ep, \beta_\ep, \gamma_\ep) \\
    &\le \ep^{-\frac{1}{r} - \frac{p-1}{p}a + \frac{1-a}{q}} \,\frac{\left( \int_{B_1} |x|^{\alpha_\ep} |\nabla u|^p \,dx \right)^{\frac{a}{p}} \left( \int_{B_1} |x|^{\beta_\ep} |u|^q \,dx \right)^{\frac{1-a}{q}}}{\left( \int_{B_1} |x|^{\gamma_\ep} |u|^r \,dx \right)^{\frac{1}{r}}}\\
    &=  \frac{\left( \int_{B_1} \frac{|y|^{\alpha_\ep} |\nabla v|^p\,dy}{\left( 1-(|y|/e)^\ep \right)^\alpha} \right)^{\frac{a}{p}} \left( \int_{B_1} \frac{|y|^{\beta_\ep } |v|^q \,dy}{\left( 1-(|y|/e)^\ep \right)^\beta} \right)^{\frac{1-a}{q}}}{\ep^{\frac{\gamma}{r} - \frac{\alpha}{p}a - \frac{\beta}{q}(1-a)} \,\left( \int_{B_1} \frac{|y|^{\gamma_\ep} |v|^r\,dy}{\left( 1-(|y|/e)^\ep \right)^\gamma}  \right)^{\frac{1}{r}}}
    \to \frac{\left( \int_{B_1} \frac{|y|^{p} |\nabla v|^p\,dy}{|y|^n\left( \log \frac{e}{|y|} \right)^\alpha} \right)^{\frac{a}{p}} \left( \int_{B_1} \frac{|v|^q\,dy}{|y|^n \left(\log \frac{e}{|y|} \right)^\beta}  \right)^{\frac{1-a}{q}}}{\left( \int_{B_1} \frac{|v|^r\,dy}{|y|^n \left( \log \frac{e}{|y|} \right)^\gamma}  \right)^{\frac{1}{r}}}
\end{align*}
as $\ep \to 0$. Therefore, we get the critical CKN inequalities (\ref{CCKN ineq}) with $\alpha >1-p$ for radial function $v$ as a limit of the original CKN inequalities (\ref{CKN ineq}). 

\subsection{CKN inequalities on balls}\label{S ball}

In this subsection, we show CKN inequalities on balls by using Theorem A. Later, we will use these inequalities in \S \ref{S n=1}. 

\begin{prop}\label{Prop CKN ball}
Let $0< a \le 1, p, q \ge 1, r>0$ and $\alpha, \beta, \gamma > -n$. 
Then there exists $C>0$ such that for any function $u \in C_0^\infty (B_R)$ the following inequality holds:
\begin{align}\label{CKN ineq on B}
    C \left( \int_{B_R} |x|^\gamma |u(x)|^r \,dx \right)^{\frac{1}{r}}
    \le \left( \int_{B_R} |x|^\alpha |\nabla u(x)|^p \,dx \right)^{\frac{a}{p}} \left( \int_{B_R} |x|^\beta |u(x)|^q \,dx \right)^{\frac{1-a}{q}}
\end{align}
if and only if the following relations hold:
\begin{align}\label{dilation valance on B}
    &\text{(Dilation balance at $0$)}\quad 
    \frac{\gamma +n}{r} \ge \frac{\alpha +n-p}{p} a \,+ \frac{\beta +n}{q}(1-a),\\
    \label{dilation at mid on B}
    &\text{(Dilation balance at a midpoint)}\quad 
    \frac{1}{r}\ge \frac{a}{p} + \frac{1-a}{q} -\frac{a}{n},\\
    \label{valance at 0 on B}
    &\text{(Balance at $0$)} \quad 
    \frac{1}{r} \le \frac{a}{p}  + \frac{1-a}{q}
    \,\,\,\text{if} \,\,\, \frac{\alpha+n -p}{p}= \frac{\gamma +n}{r}=\frac{\beta+n}{q}
\end{align}
\end{prop}


\begin{proof}[Proof of Proposition \ref{Prop CKN ball}]
We only consider the case where $a<1$, because the case where $a=1$ is easier.
We assume that the inequalities (\ref{CKN ineq on B}) hold for any $u \in C_0^\infty (B_R)$. Then we shall show (\ref{dilation valance on B})-(\ref{valance at 0 on B}). First, we consider a function $u \in C_0^\infty (B_R) \setminus \{ 0 \}$ and its scaled function $u_\la \in C_0^\infty \left(B_{R/\la} \right) \subset C_0^\infty (B_R)$ which is given by
\begin{align*}
    u_\la (x) = u(y),\,\,\text{where}\,\, 
    y= \la x\,\, \text{and} \,\, \la \in [1, \infty).
\end{align*}
Since the inequalities (\ref{CKN ineq on B}) hold for any $u_\la \in C_0^\infty (B_R)$, we have 
\begin{align*}
    0< C &\le \frac{\left( \int_{B_R} |x|^{\alpha} |\nabla u_\la|^p \,dx \right)^{\frac{a}{p}} \left( \int_{B_R} |x|^\beta |u_\la|^q \,dx \right)^{\frac{1-a}{q}}}{\left( \int_{B_R} |x|^\gamma |u_\la|^r \,dx \right)^{\frac{1}{r}}}\\
    &= \la^{\frac{n+\gamma}{r} - \left\{ \frac{n+\alpha-p}{p} a + \frac{n+\beta}{q} (1-a) \right\}} \,\frac{\left( \int_{B_R} |x|^{\alpha} |\nabla u|^p \,dx \right)^{\frac{a}{p}} \left( \int_{B_R} |x|^\beta |u|^q \,dx \right)^{\frac{1-a}{q}}}{\left( \int_{B_R} |x|^\gamma |u|^r \,dx \right)^{\frac{1}{r}}}
\end{align*}
for any $\la \in [1, \infty)$,  
where $C>0$ is independent of $\la$. If relation (\ref{dilation valance on B}) does not hold, then we can derive a contradiction by taking a limit as $\la \to \infty$ in the above inequality. Therefore, we get (\ref{dilation valance on B}). 


Second, we consider $u \in C_0^\infty (B_{R/4} \setminus \{ 0\})$ and $x_0 \in \pd B_{R/2}$. Define its scaled function $u_\la \in C_0^\infty \left( B_{\frac{R}{4\la}} (x_0) \right) \subset C_0^\infty (B_R)$ as follows.
\begin{align}\label{test mid}
    u_\la (x) = u(y),\,\,\text{where}\,\, 
   y=\la \,(x-x_0)\,\,\text{and}\,\, \la \in [1, \infty)
\end{align}
Since supp $u_\la \subset B_{\frac{3R}{4}} \setminus B_{\frac{R}{4}}=:A$, we have
\begin{align*}
    \int_{B_R} |x|^\gamma |u_\la(|x|)|^r \,dx 
    &\ge \left( \min_{x \in A}|x|^\gamma \right) \int_{A} |u_\la (x)|^r\,dx = C \la^{-n}, \\
    \int_{B_R} |x|^\alpha |\nabla u_\la (|x|)|^p \,dx 
    &\le \left( \max_{x \in A}|x|^\alpha \right) \int_{A} |\nabla u_\la (x)|^p\,dx = C \la^{p-n},\\
    \int_{B_R} |x|^\beta |u_\la (|x|)|^q \,dx
    &\le \left( \max_{x \in A}|x|^\beta \right) \int_{A} |u_\la (x)|^q\,dx = C \la^{-n}.
\end{align*}
Since the inequalities (\ref{CKN ineq on B}) hold for any $u_\la \in C_0^\infty (B_R)$, we have the inequality
$$C (\la^{-n})^{\frac{1}{r}} \le (\la^{p-n})^{\frac{a}{p}}\,(\la^{-n})^{\frac{1-a}{q}}$$
for any $\la \in [1, \infty)$, where $C>0$ is independent of $\la$. 
Therefore, we get (\ref{dilation at mid on B}). 


Third, we consider the radial function $u_\ep \in C_0(B_{R})$ for $\ep \in (0, R/4)$ as follows.
\begin{align*}
    u_\ep (|x|) =\begin{cases}
        \ep^{-\frac{n+\gamma}{r}} \quad &\text{if} \,\,\, 0\le |x| \le \ep,\\
        \left| x \right|^{-\frac{n+\gamma}{r}} \quad &\text{if} \,\,\, \ep < |x| < \frac{R}{4},\\
        \left( \frac{R}{4} \right)^{-\frac{n+\gamma}{r}-1} \left( \frac{R}{2} -|x| \right) \quad &\text{if} \,\, \frac{R}{4} \le |x| < \frac{R}{2},\\
        0 \quad &\text{if} \,\, \frac{R}{2} \le |x| < R.
    \end{cases}
\end{align*}
Then we have
\begin{align*}
    \int_{B_R} |x|^\gamma |u_\ep (|x|) |^r \,dx 
    &\ge \int_{B_{\frac{R}{4}} \setminus B_\ep} |x|^\gamma |u_\ep (|x|) |^r \,dx 
    = \w_n \log \frac{R}{4\ep},\\
    \int_{B_R} |x|^\alpha |\nabla u_\ep (|x|) |^p\,dx 
    &= \left( \frac{\gamma +n}{r} \right)^p \int_{B_{\frac{R}{4}} \setminus B_\ep} |x|^{\alpha-\frac{\gamma+n}{r}p-p} \,dx + C \\
    &=\left( \frac{\gamma +n}{r} \right)^p \w_n \int_\ep^{\frac{R}{4}} t^{-1 + p A} \,dt+C,\\
    \int_{B_R} |x|^\beta |u_\ep (|x|) |^q \,dx 
    &=\int_{B_\ep} |x|^\beta |u_\ep (|x|) |^q \,dx +  \int_{B_{\frac{R}{4}} \setminus B_\ep} |x|^\beta |u_\ep (|x|) |^q \,dx +C\\
    &= \frac{\w_n}{\beta+n} \,\ep^{\,qB}+ \w_n \int_\ep^{ \frac{R}{4}} t^{-1 + q B} \,dt+C,
\end{align*}
where $A=\frac{\alpha+n -p}{p}-\frac{\gamma+n}{r}$ and $B=\frac{\beta+n}{q}-\frac{\gamma+n}{r}$.
Since the inequalities (\ref{CKN ineq on B}) hold for any $u_\ep$, we have the inequality
\begin{align}\label{growth of ep on B}
C \left( \log \frac{R}{4\ep} \right)^{\frac{1}{r}} 
\le \left( \int_\ep^{\frac{R}{4}} t^{-1 + p A} \,dt \right)^{\frac{a}{p}}\,\left\{  \ep^{\,qB}+  \int_\ep^{ \frac{R}{4}} t^{-1 + q B} \,dt \right\}^{\frac{1-a}{q}}
\end{align}
for any small $\ep$, where $C>0$ is independent of $\ep$. 
Note that by (\ref{dilation valance on B}), we see that
\begin{align*}
    A &=\frac{\alpha+n -p}{p} - \frac{\gamma+n}{r} \le (1-a) \,\left(  \frac{\alpha+n -p}{p} - \frac{\beta+n}{q} \right),\\
    B &= \frac{\beta+n}{q} - \frac{\gamma+n}{r} \le a \,\left( \frac{\beta+n}{q} - \frac{\alpha+n -p}{p} \right).
\end{align*}
Thus, if we consider the case $A=B=0$, that is $\frac{\gamma+n}{r} = \frac{\alpha+n -p}{p}= \frac{\beta+n}{q}$, then (\ref{growth of ep on B}) becomes
\begin{align*}
C \left( \log \frac{R}{4\ep} \right)^{\frac{1}{r}} 
\le \left( \log \frac{R}{4 \ep} \right)^{\frac{a}{p} + \frac{1-a}{q}}.
\end{align*}
This implies that $\frac{1}{r} \le \frac{a}{p}+ \frac{1-a}{q}$ which is (\ref{valance at 0 on B}). 
In other cases, we do not get any restrictions from (\ref{growth of ep on B}).  
Consequently, we get (\ref{valance at 0 on B}). 


Conversely, we assume the relations (\ref{dilation valance on B})-(\ref{valance at 0 on B}). If $\frac{\gamma+n}{r}  = \frac{\alpha+n -p}{p} a + \frac{\beta+n}{q} (1-a)$, then (\ref{CKN ineq}) implies (\ref{CKN ineq on B}) directly because of $C_0^\infty (B_R) \subset C_0^\infty (\re^n)$. Thus we show the inequalities (\ref{CKN ineq on B}) for any $u \in C_0^\infty (B_R)$ when 
$\frac{\gamma+n}{r}  > \frac{\alpha+n -p}{p} a + \frac{\beta+n}{q} (1-a)$. Note that there is no upper restrictions of $\frac{1}{r}$ in \eqref{valance at 0 on B}, because $\frac{\alpha+n -p}{p} = \frac{\beta+n}{q}=\frac{\gamma +n}{r}$ is not satisfied. 
If $\frac{\beta+n}{q}\not= \frac{\gamma+n}{r}$, then we set $\tilde{\alpha} > \alpha$ as follows. 
$$
\frac{\tilde{\alpha}+n -p}{p} a + \frac{\beta+n}{q} (1-a):=\frac{\gamma+n}{r} > \frac{\alpha+n -p}{p} a + \frac{\beta+n}{q} (1-a) 
$$
By using the inequalities (\ref{CKN ineq}) in Theorem A for exponents $(\tilde{\alpha}, \beta, \gamma)$, we have
\begin{align*}
    \left( \int_{B_R} |x|^\gamma |u|^r \,dx \right)^{\frac{1}{r}} 
    &\le C \left( \int_{B_R} |x|^{\tilde{\alpha}} |\nabla u(x)|^p \,dx \right)^{\frac{a}{p}} \left( \int_{B_R} |x|^\beta |u(x)|^q \,dx \right)^{\frac{1-a}{q}}\\
    &\le C \left( \int_{B_R} |x|^\alpha |\nabla u(x)|^p \,dx \right)^{\frac{a}{p}} \left( \int_{B_R} |x|^\beta |u(x)|^q \,dx \right)^{\frac{1-a}{q}}.
\end{align*}
On the other hand, if $\frac{\beta+n}{q} = \frac{\gamma+n}{r}$, then we set $\tilde{\beta} > \beta$ as follows. 
$$
\frac{\alpha+n -p}{p} a + \frac{\tilde{\beta}+n}{q} (1-a):=\frac{\gamma+n}{r}  > \frac{\alpha+n -p}{p} a + \frac{\beta+n}{q} (1-a) 
$$
Note that $\frac{\tilde{\beta}+n}{q}\not= \frac{\gamma+n}{r}$ which implies that $\frac{\alpha+n -p}{p} = \frac{\tilde{\beta}+n}{q}=\frac{\gamma +n}{r}$ is not satisfied. By using the inequalities (\ref{CKN ineq}) for exponents $(\alpha, \tilde{\beta}, \gamma)$, we have
\begin{align*}
    \left( \int_{B_R} |x|^\gamma |u|^r \,dx \right)^{\frac{1}{r}} 
    &\le C \left( \int_{B_R} |x|^{\alpha} |\nabla u(x)|^p \,dx \right)^{\frac{a}{p}} \left( \int_{B_R} |x|^{\tilde{\beta}} |u(x)|^q \,dx \right)^{\frac{1-a}{q}}\\
    &\le C \left( \int_{B_R} |x|^\alpha |\nabla u(x)|^p \,dx \right)^{\frac{a}{p}} \left( \int_{B_R} |x|^\beta |u(x)|^q \,dx \right)^{\frac{1-a}{q}}.
\end{align*}
Therefore, we obtain Proposition \ref{Prop CKN ball}.
\end{proof}

\subsection{Nonlinear scaling}\label{S nonlinear}

A big difference between linear scaling (dilation) $x \mapsto \la x$ \,and the nonlinear scaling $x \mapsto  |x|^{\la-1} x$\, is the lack of scale invariance of integral of derivative term in (\ref{CCKN ineq}). In fact, if we consider the following scaled function $u_\la \in C_0(B_1) \cap C^\infty (B_1 \setminus \{ 0\})$ for $u \in C_0^\infty (B_1)$ and $\la \in (0, 1]$:
\begin{align*}
    u_\la (x) =\begin{cases}
    &f(\la) \,u(e^{1-\la}|x|^{\la -1}x), \quad \(x \in B_{e^{1-1/\la}} \),\\
    &0, \quad \( x \in B_1  \setminus B_{e^{1-1/\la}}\),
    \end{cases}
\end{align*}
then we have 
\begin{align*}
    &\int_{B_1} \frac{|x|^p \,|\nabla u_\la (x)|^p}{|x|^n (\log \frac{e}{|x|})^\alpha} \,dx \\
    &= \int_0^{1}\int_{\mathbb{S}^{n-1}} \left( \,\left| \frac{\partial u_\la (r\w)}{\partial r} \right|^2 + \frac{| \nabla_{\mathbb{S}^{n-1}} u_\la (r\w )|^2}{r^2} \,\right)^{\frac{p}{2}} r^{p-1} \left( \log \frac{e}{r} \right)^{-\alpha}\,drdS_\w  \\
    &= f(\la)^p \iint \left( \,\left| \frac{\partial u (s\w)}{\partial s} \right|^2 + \left(\frac{ds}{dr} r\right)^{-2} | \nabla_{\mathbb{S}^{n-1}} u (s\w )|^2 \,\right)^{\frac{p}{2}} \left( \frac{ds}{dr} r \right)^{p-1} \left( \frac{1}{\la} \log \frac{e}{s} \right)^{-\alpha}\,dsdS_\w  \\
    &= f(\la)^p \la^{\alpha -1+p}\iint \left( \,\left| \frac{\partial u (s\w)}{\partial s} \right|^2 +  \frac{| \nabla_{\mathbb{S}^{n-1}} u (s\w )|^2}{\la^2 s^2} \,\right)^{\frac{p}{2}} s^{p-1} \left( \log \frac{e}{s} \right)^{-\alpha}\,dsdS_\w  \\
    &\ge f(\la)^p \la^{\alpha -1+p} \int_{B_1} \frac{|y|^p \,|\nabla u(y)|^p}{|y|^n (\log \frac{e}{|y|})^\alpha} \,dy.
\end{align*}
Although the derivative term has scale invariance when $u$ is a radial function and $f(\la) = \la^{-\frac{\alpha-1+p}{p}}$, there is no scale invariance when $u$ is a non-radial function. 

Scale invariance plays an important role to study functional inequalities and PDEs. We emphasize that our inequalities (\ref{CCKN ineq}) do not have scale invariance in general, but scaling argument by \cite{CKN} still works even for the nonlinear scaling. For the details, see the proof of Lemma \ref{lem:mean zero a=1} in [Step2] in \S \ref{S a=1}, \S \ref{S a<1 n>1} and \S \ref{S a<1 n>1 last}. 

%
%
\section{Necessity part of Theorem \ref{Thm CCKN}}\label{S Necessity}

In this section, we show that if the inequality (\ref{CCKN ineq}) holds for any $u \in C_0^\infty (B_R)$, then conditions (\ref{n-dilation valance})-(\ref{n-valance at 0}) hold. First, we consider a radial function $u \in C_0^\infty (B_R \setminus \{ 0 \})$ and its nonlinear scaled function 
$$u_\la \in C^\infty \left(B_{e^{1-\frac{1}{\la}} R} \setminus \{ 0\} \right) \cap C_0 \left(B_{e^{1-\frac{1}{\la}} R} \right) \subset C_0 (B_R)$$ which is given by
\begin{align*}
    u_\la (|x|) = u(|y|),\,\,\text{where}\,\, 
    \frac{|y|}{eR} = \left(\frac{|x|}{eR}\right)^\la\,\,\text{and}\,\, \la \in (0, 1].
\end{align*}
In the same way as it in \S \ref{S nonlinear}, we have
\begin{align*}
    \int_{B_R} \frac{|u_\la(|x|)|^r}{|x|^n (\log \frac{eR}{|x|})^\gamma} \,dx 
    &= \la^{\gamma -1} \int_{B_R} \frac{|u(|y|)|^r}{|y|^n (\log \frac{eR}{|y|})^\gamma} \,dy,\\
    \int_{B_R} \frac{|x|^p \,|\nabla u_\la (|x|)|^p}{|x|^n (\log \frac{eR}{|x|})^\alpha} \,dx 
    &=\la^{\alpha-1+p} \int_{B_R} \frac{|y|^p \,|\nabla u(|y|)|^p}{|y|^n (\log \frac{eR}{|y|})^\alpha} \,dy,\\
    \int_{B_R} \frac{|u_\la (|x|)|^q}{|x|^n (\log \frac{eR}{|x|})^\beta} \,dx
    &=\la^{\beta -1} \int_{B_R} \frac{|u(|y|)|^q}{|y|^n (\log \frac{eR}{|y|})^\beta} \,dy.
\end{align*}
Since the inequalities (\ref{CCKN ineq}) hold for any $u_\la$, we have the inequality
$$C (\la^{\gamma -1})^{\frac{1}{r}} \le (\la^{p-1+\alpha})^{\frac{a}{p}}\,(\la^{\beta -1})^{\frac{1-a}{q}}$$
for any $\la \in (0, 1]$, where $C>0$ is independent of $\la$. Therefore, we get (\ref{n-dilation valance}).

Second, we consider the same test function $u_\la \in C_0^\infty (B_R)$ which is given by (\ref{test mid}). In a similar way to \S \ref{S ball}, we get (\ref{n-valance at mid}).

Third, we consider the radial function $u_\ep \in C_0(B_{R/2})$ for $\ep \in (0, R/e)$ as follows.
\begin{align*}
    u_\ep (|x|) =\begin{cases}
        \left( \log \frac{eR}{\ep} \right)^{\frac{\gamma -1}{r}} \quad &\text{if} \,\,\, 0\le |x| \le \ep,\\
        \left( \log \frac{eR}{|x|} \right)^{\frac{\gamma -1}{r}} \quad &\text{if} \,\,\, \ep < |x| < \frac{R}{e},\\
        \left( R - 2|x| \right) \,2^{\frac{\gamma-1}{2}} \frac{e}{R(e-2)} \quad &\text{if} \,\, \frac{R}{e} \le |x| < \frac{R}{2}.\\
    \end{cases}
\end{align*}
Then we have
\begin{align*}
    \int_{B_R} \frac{|u_\ep(|x|)|^r}{|x|^n (\log \frac{eR}{|x|})^\gamma} \,dx 
    &\ge \int_{B_{\frac{R}{e}} \setminus B_\ep} \frac{|u_\ep(|x|)|^r}{|x|^n (\log \frac{eR}{|x|})^\gamma} \,dx 
    = \w_n \left( \log \log \frac{eR}{\ep} - \log 2 \right),\\
    \int_{B_R} \frac{|x|^p \,|\nabla u_\ep (|x|)|^p}{|x|^n (\log \frac{eR}{|x|})^\alpha} \,dx 
    &= \left( \frac{\gamma -1}{r} \right)^p \int_{B_{\frac{R}{e}} \setminus B_\ep}\frac{(\log \frac{eR}{|x|})^{\frac{\gamma -1}{r}p -p}}{|x|^n (\log \frac{eR}{|x|})^\alpha} \,dx + C \\
    &=\left( \frac{\gamma -1}{r} \right)^p \w_n \int_2^{\log \frac{eR}{\ep}} t^{-1 + p A} \,dt+C,\\
    \int_{B_R} \frac{|u_\ep(|x|)|^q}{|x|^n (\log \frac{eR}{|x|})^\beta} \,dx 
    &=\int_{B_\ep} \frac{|u_\ep(|x|)|^q}{|x|^n (\log \frac{eR}{|x|})^\beta} \,dx +  \int_{B_{\frac{R}{e}} \setminus B_\ep} \frac{|u_\ep(|x|)|^q}{|x|^n (\log \frac{eR}{|x|})^\beta} \,dx +C\\
    &= \frac{\w_n}{\beta -1} \left( \log \frac{eR}{\ep} \right)^{\frac{\gamma -1}{r}q + 1-\beta}+ \w_n \int_2^{\log \frac{eR}{\ep}} t^{-1 + q B} \,dt+C,
\end{align*}
where $A=\frac{\gamma -1}{r} - \frac{\alpha -1+p}{p}$ and $B= \frac{\gamma -1}{r} - \frac{\beta-1}{q}$. 
Since the inequality (\ref{CCKN ineq}) holds for any $u_\ep$, we have the inequality
\begin{align}\label{growth of ep}
C \left( \log \log \frac{eR}{\ep} \right)^{\frac{1}{r}} 
\le \left( \int_2^{\log \frac{eR}{\ep}} t^{-1 + p A} \,dt \right)^{\frac{a}{p}}\,\left\{  \left( \log \frac{eR}{\ep} \right)^{qB} + \int_2^{\log \frac{eR}{\ep}} t^{-1 +qB} \,dt \right\}^{\frac{1-a}{q}}
\end{align}
for any $\ep \in (0, R/e)$, where $C>0$ is independent of $\ep$. 
Note that by (\ref{n-dilation valance}), we see that
\begin{align*}
    A &=\frac{\gamma -1}{r} - \frac{\alpha -1+p}{p} \ge (1-a) \,\left(  \frac{\beta -1}{q} - \frac{\alpha -1+p}{p} \right),\\
    B &= \frac{\gamma -1}{r} - \frac{\beta-1}{q} \ge a \,\left( \frac{\alpha -1+p}{p} - \frac{\beta -1}{q} \right).
\end{align*}
Thus, if we consider the case $A=B=0$, that is $\frac{\gamma -1}{r} = \frac{\alpha -1+p}{p}= \frac{\beta -1}{q}$, then (\ref{growth of ep}) becomes
\begin{align*}
C \left( \log \log \frac{eR}{\ep} \right)^{\frac{1}{r}} 
\le \left( \log \log \frac{eR}{\ep} \right)^{\frac{a}{p} + \frac{1-a}{q}}.
\end{align*}
This implies $\frac{1}{r} \le \frac{a}{p}+ \frac{1-a}{q}$ which is (\ref{n-valance at 0}). 
In other cases, we do not get any restrictions from (\ref{growth of ep}).  
Consequently, we get (\ref{n-valance at 0}). 
Therefore \eqref{n-dilation valance}-\eqref{n-valance at 0} are necessary conditions of (\ref{CCKN ineq}).  

%
%
\section{Sufficiency part of Theorem \ref{Thm CCKN} when $a=1$ and $n \ge 2$}\label{S a=1}

In this section, we show critical CKN inequalities (\ref{CCKN ineq}) with $a=1, n \ge 2$ and explain key ideas to show (\ref{CCKN ineq}), especially {\it nonlinear scaling argument}, 
which is a development of scaling argument by Caffarelli-Kohn-Nirenberg \cite{CKN}. Original scaling argument for Theorem A is a way to create polynomial weighted inequalities (\ref{CKN ineq}) from non-weighted inequalities (GN inequalities). In this paper, we develop this argument into it for nonlinear scaling and create logarithmic weighted inequalities (\ref{CCKN ineq}) from GN inequalities. For the details, see the proof of Lemma \ref{lem:mean zero a=1}. 
Note that sufficiency part of Theorem \ref{Thm CCKN} with $a=1, n \ge 2$ becomes the following.

\begin{theorem}\label{Thm a=1 CCKN}
Let $n \ge 2, p \ge 1, r>0, \alpha \in \re$ and $\gamma >1$ satisfy either {\rm (i)} or {\rm (ii):} 
\begin{align*}
  &{\rm (i)}\,\, p \le r \, 
   \begin{cases}
   \le \frac{np}{n-p} \,\, &(1 \le p < n ),  \\
   < \infty \,\, &(p \ge  n ),
    \end{cases} \quad
    \gamma \ge \frac{\alpha -1+p}{p} r +1,\\
    &{\rm (ii)}\,\,r < p, \quad \gamma > \frac{\alpha-1+p}{p} r +1.
\end{align*}
Then there exists $C>0$ such that for any functions $u \in C_0^\infty (B_R)$ the following inequality holds.
\begin{align}\label{a=1 CCKN ineq}
C \left( \int_{B_R} \frac{|u(x)|^r}{|x|^n (\log \frac{eR}{|x|})^\gamma} \,dx \right)^{\frac{p}{r}} \le \int_{B_R} \frac{|x|^p \,|\nabla u(x) |^p}{|x|^n (\log \frac{eR}{|x|})^\alpha} \,dx
\end{align}
\end{theorem}

Weighted critical Hardy inequalities (\ref{a=1 CCKN ineq}) are already shown by \cite{HK} ($\alpha =0$), \cite{H,Sano-TF(JGA)} ($p=n=2, \alpha >-1$), \cite{H,AHN,H p=1} ($C_0^\infty (B_R)\setminus \{ 0\}$) and \cite{P} based on generalized rearrangement technique, nonlinear potential theory \cite{Adams}, generalized isoperimetric inequality or harmonic transplantation technique. 
In this paper, without those technique, we give a self-contained proof of the inequalities (\ref{a=1 CCKN ineq}) based on nonlinear scaling argument.


\begin{proof}[Proof of Theorem \ref{Thm a=1 CCKN}] Without loss of generality, we can assume $R=1$ because using dilation $\tilde{u}(x)= u(Rx)$, we have
\begin{align*}
    \int_{B_R} \frac{|u(y)|^r}{|y|^n (\log \frac{eR}{|y|})^\gamma} \,dy&= \int_{B_1} \frac{|\tilde{u}(x)|^r}{|x|^n (\log \frac{e}{|x|})^\gamma} \,dx,\\
    \int_{B_R} \frac{|y|^p |\nabla u(y)|^p}{|y|^n (\log \frac{eR}{|y|})^\alpha} \,dy&= \int_{B_1} \frac{|x|^p |\nabla \tilde{u}(x)|^p}{|x|^n (\log \frac{e}{|x|})^\alpha} \,dx.
\end{align*}
Also, we consider the case (i) only because in the case (ii), we can reduce it to (i) by using the H\"older inequality. In fact, we assume that the case (i) is already obtained and we consider the case (ii), that is, $0<r<p$ and $\gamma > \max\{ 1, \gamma_\alpha \}$, where $\gamma_\alpha := \frac{\alpha -1+p}{p}r +1$. 
If $\alpha >1-p$, then 
\begin{align*}
    \int_{B_1} \frac{|u(x)|^r \,dx}{|x|^n (\log \frac{e}{|x|})^\gamma}  
    &\le \left(  \int_{B_1} \frac{|u(x)|^p\,dx}{|x|^n (\log \frac{e}{|x|})^{\alpha+p}} \right)^{\frac{r}{p}} \left( \int_{B_1} |x|^{-n} \left(\log \frac{e}{|x|}\right)^{-\left( \gamma - \frac{\alpha +p}{p} r \right) \frac{p}{p-r}} \,dx \right)^{1-\frac{r}{p}}\\
    &\le C\left(  \int_{B_1} \frac{|u(x)|^p}{|x|^n (\log \frac{e}{|x|})^{\alpha +p}} \,dx\right)^{\frac{r}{p}}
    \le C\left(  \int_{B_1} \frac{|x|^p |\nabla u(x)|^p}{|x|^n (\log \frac{e}{|x|})^{\alpha}} \,dx\right)^{\frac{r}{p}}
\end{align*}
by using $\gamma > \gamma_\alpha$ and (i). 
On the other hand, if $\alpha \le 1-p$, then 
$$\gamma >1\ge \max\{ \alpha+p, \gamma_\alpha\}.$$
Therefore, we have
\begin{align*}
    \int_{B_1} \frac{|u(x)|^r \,dx}{|x|^n (\log \frac{e}{|x|})^\gamma}  
    &\le \left(  \int_{B_1} \frac{|u(x)|^p\,dx}{|x|^n (\log \frac{e}{|x|})^{\gamma }} \right)^{\frac{r}{p}} \left( \int_{B_1} \frac{1}{|x|^n (\log \frac{e}{|x|})^{\gamma }} \,dx \right)^{1-\frac{r}{p}}\\
    &\le C\left(  \int_{B_1} \frac{|u(x)|^p}{|x|^n (\log \frac{e}{|x|})^{\gamma}} \,dx\right)^{\frac{r}{p}}
    \le C\left(  \int_{B_1} \frac{|x|^p |\nabla u(x)|^p}{|x|^n (\log \frac{e}{|x|})^{\alpha}} \,dx\right)^{\frac{r}{p}}
\end{align*}
by using $\gamma >\max\{ 1, \alpha+p\}=1$ and (i).
Thus, we consider the case (i) and $R=1$.

\noindent
{\bf [Step 1: Radial functions]} First, we will show the inequalities (\ref{a=1 CCKN ineq}) for any radial functions $u \in C_{0}^\infty (B_1)$ by using weighted critical Hardy inequalities and radial lemma in the followings. We will show them later. 

\begin{lemma}\label{lem:Critical Hardy}(Weighted critical Hardy inequalities) Let $\tilde{\gamma} >1$ and $p \ge 1$. Then for any $u \in C_0^\infty(B_1)$, the following inequalities hold.
    \begin{align}\label{Critical Hardy}
        &\left( \frac{\tilde{\gamma}-1}{p} \right)^p  \int_{B_1} \frac{|u(x)|^p}{|x|^n (\log \frac{e}{|x|})^{\tilde{\gamma}}} \,dx  
        \le \int_{B_1} \frac{|x|^p \,|\nabla u(x) |^p}{|x|^n (\log \frac{e}{|x|})^{\tilde{\gamma}-p}} \,dx\\
        \label{Critical Hardy log}
        &\left( \frac{\tilde{\gamma}-1}{p} \right)^p  \int_{B_1} \frac{|u(x)|^p\,dx}{|x|^n \left( \log \frac{e}{|x|}\right) \( \log \log \frac{e}{|x|} \)^{\tilde{\gamma}}} 
        \le \int_{B_1} \frac{|x|^p \,|\nabla u(x) |^p\,dx}{|x|^n \left( \log \frac{e}{|x|} \right)^{1-p} \( \log \log \frac{e}{|x|} \)^{\tilde{\gamma}-p} } 
    \end{align}
\end{lemma}

\begin{lemma}\label{lem:Radial}(Radial lemma)
    For any radial functions $u \in C_0^\infty (B_1)$ and $x \in B_1 \setminus \{ 0\}$, the following estimates hold.
    \begin{align*}
        |u(x)| \le \left( \frac{1}{\w_n} \int_{B_1} \frac{|x|^p \,|\nabla u|^p\,dx}{|x|^n (\log \frac{e}{|x|})^{\alpha}} \right)^{\frac{1}{p}}  
        \begin{cases}
        \max\limits_{|x| \le s \le 1}\( \log \frac{e}{s}\)^\alpha  &\text{if}\,\, p=1,\\
            \left( \frac{p-1}{|\alpha +p -1|} \right)^{\frac{p-1}{p}} \left| \left( \log \frac{e}{|x|} \right)^{\frac{\alpha -1+p}{p-1}} -1\right|^{\frac{p-1}{p}} \hspace{-1em} &\text{if}\,\, \alpha \not= 1-p, p>1,\\
            \left(\log \log \frac{e}{|x|} \right)^{\frac{p-1}{p}}  &\text{if}\,\, \alpha =1-p, p>1
        \end{cases}
    \end{align*}
\end{lemma}

\noindent
If $\alpha >1-p$, by using (i), (\ref{Critical Hardy}) and Lemma \ref{lem:Radial}, we have
\begin{align*}
    \int_{B_1} \frac{|u(x)|^r }{|x|^n (\log \frac{e}{|x|})^\gamma} \,dx
    &\le C
    \int_{B_1} \frac{|u(x)|^r }{|x|^n (\log \frac{e}{|x|})^{\frac{\alpha-1+p}{p}r +1}}\,dx\\
    &\le C \left( \underset{x \in B_1}{\rm ess\,sup} \,\,\frac{|u(x)|}{\left( \log \frac{e}{|x|} \right)^{\frac{\alpha-1+p}{p}}} \right)^{r-p} \int_{B_1} \frac{|u(x) |^p}{|x|^n (\log \frac{e}{|x|})^{\alpha+p}} \,dx \\
    &\le C \left( \int_{B_1} \frac{|x|^p \,|\nabla u(x) |^p}{|x|^n (\log \frac{e}{|x|})^{\alpha}} \,dx \right)^{\frac{r}{p}}.
\end{align*}
If $\alpha <1-p\,(<\gamma -p)$, by using (\ref{Critical Hardy}) and Lemma \ref{lem:Radial}, we have
\begin{align*}
    \int_{B_1} \frac{|u(x)|^r }{|x|^n (\log \frac{e}{|x|})^\gamma} \,dx
    &\le \left( \underset{x \in B_1}{\rm ess\,sup} \,\,|u(x)| \right)^{r-p}
    \int_{B_1} \frac{|u(x)|^p }{|x|^n (\log \frac{e}{|x|})^{\gamma}}\,dx\\
    &\le C  \left( \int_{B_1} \frac{|x|^p \,|\nabla u(x) |^p}{|x|^n (\log \frac{e}{|x|})^{\alpha}} \,dx \right)^{\frac{r-p}{p}} \int_{B_1} \frac{|x|^p \,|\nabla u(x) |^p}{|x|^n (\log \frac{e}{|x|})^{\gamma-p}} \,dx \\
    &\le C  \left( \int_{B_1} \frac{|x|^p \,|\nabla u(x) |^p}{|x|^n (\log \frac{e}{|x|})^{\alpha}} \,dx \right)^{\frac{r}{p}}.
\end{align*}
If $\alpha =1-p$, by using (\ref{Critical Hardy log}) and Lemma \ref{lem:Radial}, we have
\begin{align*}
    \int_{B_1} \frac{|u(x)|^r }{|x|^n (\log \frac{e}{|x|})^\gamma} \,dx
    &\le \left( \underset{x \in B_1}{\rm ess\,sup} \,\,\frac{|u(x)|}{\left( \log \log \frac{e}{|x|} \right)^{\frac{p-1}{p}}} \right)^{r-p}
    \int_{B_1} \frac{|u(x)|^p \left( \log \log \frac{e}{|x|} \right)^{\frac{p-1}{p}(r-p)}}{|x|^n (\log \frac{e}{|x|})^{\gamma}}\,dx\\
    &\le C  \left( \int_{B_1} \frac{|x|^p \,|\nabla u(x) |^p}{|x|^n (\log \frac{e}{|x|})^{\alpha}} \,dx \right)^{\frac{r-p}{p}} \int_{B_1} \frac{|u(x) |^p}{|x|^n \left( \log \frac{e}{|x|} \right) \( \log \log \frac{e}{|x|} \)^{p} } \,dx \\
    &\le C  \left( \int_{B_1} \frac{|x|^p \,|\nabla u(x) |^p}{|x|^n (\log \frac{e}{|x|})^{\alpha}} \,dx \right)^{\frac{r}{p}},
\end{align*}
where the second inequality comes from the boundedness of the function:
$$\underset{x \in B_1}{\rm ess\,sup} \,\,\left( \log \frac{e}{|x|}\right)^{1-\gamma} \left( \log \log \frac{e}{|x|}\right)^{\frac{p-1}{p}r+1}< \infty$$
Therefore, we get the inequalities (\ref{a=1 CCKN ineq}) for any radial functions $u \in C_{0}^\infty (B_1)$. 

\noindent
{\bf [Step 2: Spherical mean zero functions]} For $k \in \N$, we set 
\begin{align*}
    A_k=\left\{ x \in \re^n \,\,\Biggr|\,\, e \left( \frac{1}{2e}\right)^{2^k} <|x| \le e \left( \frac{1}{2e}\right)^{2^{k-1}}\right\}\,\,\text{and}\,\, A_0=B_{1} \setminus \overline{B_{1/2}}.
\end{align*}
First, we claim that the inequalities (\ref{a=1 CCKN ineq}) hold on the annulus $A_k$ for any spherical mean zero functions as follows. We will show Lemma \ref{lem:mean zero a=1} later. 

\begin{lemma}\label{lem:mean zero a=1}
Suppose the condition {\rm (i)} in Theorem \ref{Thm a=1 CCKN}. Then 
    there exists $C>0$ such that for any $k \in \N$ and any $u \in C^\infty (B_1)$ with $\int_{\mathbb{S}^{n-1}} u(t\w) \,dS_\w =0$ for any $t \in [0, 1)$, 
    the following inequality holds.
    \begin{align}\label{CCKN A_k}
     \int_{A_k} \frac{|u(x)|^r }{|x|^n (\log \frac{e}{|x|})^\gamma} \,dx
    \le C  \left( \int_{A_k} \frac{|x|^p \,|\nabla u(x) |^p}{|x|^n (\log \frac{e}{|x|})^{\alpha}} \,dx \right)^{\frac{r}{p}}
\end{align}
\end{lemma}
Note that if $k=0$, then we also obtain the inequalities (\ref{CCKN A_k}) on $A_0$ for any spherical mean zero functions $u \in C^\infty (B_1)$ 
in the same way.  
If we get Lemma \ref{lem:mean zero a=1}, then we get the inequalities (\ref{a=1 CCKN ineq}) on $B_1$ for any spherical mean zero functions $u$, that is $\int_{\mathbb{S}^{n-1}} u(t\w) \,dS_\w =0$ for any $t \in [0, 1)$. In fact, since $B_1 \setminus \{ 0\} = \bigcup_{k=0}^\infty A_k$, taking $\sum_{k=0}^\infty$ on (\ref{CCKN A_k}), we have
\begin{align*}
\int_{B_1} \frac{|u(x)|^r }{|x|^n (\log \frac{e}{|x|})^\gamma} \,dx
&= \sum_{k=0}^\infty
     \int_{A_k} \frac{|u(x)|^r }{|x|^n (\log \frac{e}{|x|})^\gamma} \,dx\\
    &\le C  \sum_{k=0}^\infty \left( \int_{A_k} \frac{|x|^p \,|\nabla u(x) |^p}{|x|^n (\log \frac{e}{|x|})^{\alpha}} \,dx \right)^{\frac{r}{p}}
    \le C  \left( \int_{B_1} \frac{|x|^p \,|\nabla u(x) |^p}{|x|^n (\log \frac{e}{|x|})^{\alpha}} \,dx \right)^{\frac{r}{p}},
\end{align*}
where the last inequality comes from $\frac{r}{p} \ge 1$ and the inequality 
\begin{align*}
    \left( \,\sum_{k=0}^\infty |a_k|^c \,\right)^{\frac{1}{c}} \le \sum_{k=0}^\infty |a_k| < \infty 
\end{align*}
holds for $c \ge 1$ and $\{ a_k \}_{k=0}^\infty \subset \re$. 
Therefore, we get the inequalities (\ref{a=1 CCKN ineq}) for any spherical mean zero functions $u \in C^\infty (B_1)$.


\noindent
{\bf [Step 3: Any functions]} Finally, we show (\ref{a=1 CCKN ineq}) for any functions $u \in C_0^\infty (B_1)$ by using [Step 1] and [Step 2]. For $u$, we consider its spherical average
\begin{align*}
    U(x)=U(t)=\frac{1}{\w_n} \int_{\mathbb{S}^{n-1}} u(t\w) \,dS_\w \quad (t=|x|).
\end{align*}
By using Jensen's inequality, we have
\begin{align*}
    |U'(t)|^p \le \left( \frac{1}{\w_n} \int_{\mathbb{S}^{n-1}} \left| \frac{\partial u}{\partial t} (t\w) \right| \,dS_\w\right)^p
    \le \frac{1}{\w_n} \int_{\mathbb{S}^{n-1}} \left| \frac{\partial u}{\partial t} (t\w) \right|^p \,dS_\w,
\end{align*}
which implies 
\begin{align}\label{nabla U}
    \int_{B_1} \frac{|x|^p \,|\nabla U |^p}{|x|^n (\log \frac{e}{|x|})^{\alpha}} \,dx 
    &\le \w_n \int_0^1 |U'(t)|^p t^{p-1} \left( \log \frac{e}{t} \right)^{-\alpha} \,dt \notag \\
    &\le \int_0^1 \int_{\mathbb{S}^{n-1}} \left| \frac{\partial u}{\partial t} (t\w) \right|^p t^{p-1} \left( \log \frac{e}{t} \right)^{-\alpha} \,dt \,dS_\w \notag \\
    &\le \int_{B_1} \frac{|x|^p \,|\nabla u |^p}{|x|^n (\log \frac{e}{|x|})^{\alpha}} \,dx.
\end{align}
Since $u-U$ is a spherical mean zero function, we can apply [Step 2] for $u-U$. Thus, by using the triangle inequality, [Step 2] and (\ref{nabla U}), we have
\begin{align*}
     &\left( \int_{B_1} \frac{|u|^r \,dx}{|x|^n (\log \frac{e}{|x|})^\gamma} \right)^{\frac{1}{r}} -\left( \int_{B_1} \frac{|U|^r \,dx}{|x|^n (\log \frac{e}{|x|})^\gamma} \right)^{\frac{1}{r}}
     \le \left( \int_{B_1} \frac{|u-U|^r \,dx}{|x|^n (\log \frac{e}{|x|})^\gamma} \right)^{\frac{1}{r}} \\
    &\le C  \left( \int_{B_1} \frac{|x|^p \,|\nabla u- \nabla U |^p\,dx}{|x|^n (\log \frac{e}{|x|})^{\alpha}}  \right)^{\frac{1}{p}}\\
    &\le C\left( \int_{B_1} \frac{|x|^p \,|\nabla u |^p\,dx}{|x|^n (\log \frac{e}{|x|})^{\alpha}}  \right)^{\frac{1}{p}} + C\left( \int_{B_1} \frac{|x|^p \,|\nabla U |^p\,dx}{|x|^n (\log \frac{e}{|x|})^{\alpha}} \right)^{\frac{1}{p}}
    \le C \left( \int_{B_1} \frac{|x|^p \,|\nabla u |^p\,dx}{|x|^n (\log \frac{e}{|x|})^{\alpha}}  \right)^{\frac{1}{p}}.
\end{align*}
Therefore, by using [Step 1] for $U$ and (\ref{nabla U}), we have
\begin{align*}
    \left( \int_{B_1} \frac{|u|^r \,dx}{|x|^n (\log \frac{e}{|x|})^\gamma} \right)^{\frac{1}{r}} 
    &\le C \left( \int_{B_1} \frac{|x|^p \,|\nabla u |^p\,dx}{|x|^n (\log \frac{e}{|x|})^{\alpha}}  \right)^{\frac{1}{p}}+ \left( \int_{B_1} \frac{|U|^r \,dx}{|x|^n (\log \frac{e}{|x|})^\gamma} \right)^{\frac{1}{r}}\\
    &\le C \left( \int_{B_1} \frac{|x|^p \,|\nabla u |^p\,dx}{|x|^n (\log \frac{e}{|x|})^{\alpha}}  \right)^{\frac{1}{p}}+ C \left( \int_{B_1} \frac{|x|^p \,|\nabla U |^p\,dx}{|x|^n (\log \frac{e}{|x|})^{\alpha}}  \right)^{\frac{1}{p}}\\
    &\le C \left( \int_{B_1} \frac{|x|^p \,|\nabla u |^p\,dx}{|x|^n (\log \frac{e}{|x|})^{\alpha}}  \right)^{\frac{1}{p}}. 
\end{align*}
Hence, we complete the proof of Theorem \ref{Thm a=1 CCKN}. 
\end{proof}

\begin{proof}[Proof of Lemma \ref{lem:Critical Hardy}]
By using integration by parts and the H\"older inequality, we have 
\begin{align*}
   (\tilde{\gamma }-1) \int_{B_1} \frac{|u(x)|^p \,dx}{|x|^n (\log \frac{e}{|x|})^{\tilde{\gamma}}} 
   &= 
   \int_{B_1} |u(x)|^p \,{\rm div} \left( \frac{x}{|x|^n(\log \frac{e}{|x|})^{\tilde{\gamma} -1}} \right) dx\\
   &\le p \int_{B_1} \frac{|u(x)|^{p-1} |\nabla u(x)|}{|x|^{n-1} (\log \frac{e}{|x|})^{\tilde{\gamma} -1}} \,dx\\
   &\le p \left( \int_{B_1} \frac{|x|^{p} |\nabla u(x)|^p \,dx}{|x|^n (\log \frac{e}{|x|})^{\tilde{\gamma}-p}} \right)^{\frac{1}{p}} \left( \int_{B_1} \frac{|u(x)|^p \,dx}{|x|^n (\log \frac{e}{|x|})^{\tilde{\gamma}}} \right)^{1-\frac{1}{p}}
\end{align*}
    which imply (\ref{Critical Hardy}). By using 
    \begin{align*}
        {\rm div} \left( \frac{x}{|x|^n (\log \log \frac{e}{|x|})^{\tilde{\gamma}-1}} \right) = \frac{\tilde{\gamma}-1}{|x|^n (\log \frac{e}{|x|}) (\log \log \frac{e}{|x|})^{\tilde{\gamma}}},
    \end{align*}
    we can also show (\ref{Critical Hardy log}) in the same way. 
\end{proof}

\begin{proof}[Proof of Lemma \ref{lem:Radial}]
We show only the case where $\alpha >1-p$ and $p>1$. For any radial functions $u \in C_0^\infty (B_1)$ and $x \in B_1 \setminus \{ 0\}$, we have 
    \begin{align*}
        |u(|x|)| &\le \int_{|x|}^1 |u'(s)|\,ds 
        \le \left( \int_{|x|}^1 \frac{|u'(s)|^p s^{p-1}}{(\log \frac{e}{s})^\alpha} \,ds \right)^{\frac{1}{p}} \left( \int_{|x|}^1 s^{-1} \left( \log \frac{e}{s} \right)^{\frac{\alpha}{p-1}} \,ds \right)^{\frac{p-1}{p}}\\
        &\le \left( \frac{1}{\w_n} \int_{B_1 \setminus B_{|x|}} \frac{|y|^p \,|\nabla u(y) |^p}{|y|^n (\log \frac{e}{|y|})^{\alpha}} \,dy \right)^{\frac{1}{p}}\left( \frac{p-1}{\alpha + p-1} \right)^{\frac{p-1}{p}}
        \left( \left( \log \frac{e}{|x|} \right)^{\frac{\alpha +p-1}{p-1} } -1\right)^{\frac{p-1}{p}}
    \end{align*}
    which imply Lemma \ref{lem:Radial}. 
\end{proof}

\begin{proof}[Proof of Lemma \ref{lem:mean zero a=1}]
First, we consider the case where $k=1$. 
Note that 
$$\int_{A_1} u(x) \,dx =\int_{1/4e}^{1/2} \int_{\mathbb{S}^{n-1}} u(r\w) \,dS_\w \,r^{n-1} dr =0$$ because of spherical mean zero condition of $u$.  
By using the Sobolev inequality: 
$$\| u\|_{L^r(A_1)} \le C \,\| u \|_{W^{1, p}(A_1)}$$ for any functions $u \in C^\infty (A_1)$ and the Poincar\'e inequality:
$$\| u\|_{L^p(A_1)} \le C \,\| \nabla u \|_{L^{p}(A_1)}$$
for any functions $u \in C^{\infty}(A_1)$ with $\int_{A_1} u(x) \,dx=0$ (see e.g. \cite[p.290 Theorem 1]{Evans}), we have
\begin{align}\label{A_1}
     \int_{A_1} \frac{|u(x)|^r }{|x|^n (\log \frac{e}{|x|})^\gamma} \,dx
     &\le C \int_{A_1} |u(x)|^r\,dx \notag \\
     &\le C \left( \int_{A_1} |\nabla u(x)|^p + |u(x)|^p \,dx \right)^{\frac{r}{p}} \notag \\
     &\le C \left( \int_{A_1} |\nabla u(x)|^p  \right)^{\frac{r}{p}} \notag \\
    &\le C  \left( \int_{A_1} \frac{|x|^p \,|\nabla u(x) |^p}{|x|^n (\log \frac{e}{|x|})^{\alpha}} \,dx \right)^{\frac{r}{p}},
\end{align}
where the first and the last inequalities come from the boundedness of weight functions $|x|^{-n} \left( \log \frac{e}{|x|} \right)^{-\gamma}$ and $|x|^{p-n} \left( \log \frac{e}{|x|} \right)^{-\alpha}$ on $A_1$. 
For $k \ge 2$, we consider the nonlinear scaling in \S \ref{S nonlinear} as follows.
\begin{align*}
    u_k (y) = u(x), \,\text{where}\,\, y= e^{1-\la_k} |x|^{\la_k -1}x \,\,\text{and}\,\, \la_k = 2^{-k+1}(\le 1)
\end{align*}
Note that $u_k$ is also a spherical mean zero function on $A_1$ because $u$ is a spherical mean zero function on $A_k$. 
By using the condition (i) in Theorem \ref{Thm a=1 CCKN}, (\ref{A_1}) and the calculation in \S \ref{S nonlinear}, we have 
\begin{align*}
    \left( \int_{A_k} \frac{|u(x)|^r \,dx}{|x|^n (\log \frac{e}{|x|})^\gamma} \right)^{\frac{1}{r}}
    &= \la_k^{\frac{\gamma -1}{r}} \left( \int_{A_1} \frac{|u_k(y)|^r \,dy}{|y|^n (\log \frac{e}{|y|})^\gamma} \right)^{\frac{1}{r}}\\
    &\le C \la_k^{\frac{p+\alpha -1}{p}} \left( \int_{A_1} \frac{|y|^p \,|\nabla u_k(y) |^p\,dy}{|y|^n (\log \frac{e}{|y|})^{\alpha}} \right)^{\frac{1}{p}}\\
    &\le C \left( \int_{A_k} \frac{|x|^p \,|\nabla u(x) |^p}{|x|^n (\log \frac{e}{|x|})^{\alpha}} \,dx \right)^{\frac{1}{p}}.
\end{align*}
Therefore, we complete the proof of Lemma \ref{lem:mean zero a=1}. 
\end{proof}

%
%

\section{Sufficiency part of Theorem \ref{Thm CCKN} when $n =1$}\label{S n=1}

In this section, we show the inequalities (\ref{CCKN ineq}) when $n=1$. Thanks to the reflection, it is enough to show the inequalities (\ref{CCKN ineq}) on the interval $[0, R)$.

\begin{theorem}\label{Thm CCKN n=1}
Let $0< a \le 1, p, q \ge 1, r>0, \alpha \in \re$ and $\beta, \gamma >1$. If the following conditions hold:
\begin{align*}
    &{\rm (i)}\quad \frac{\gamma -1}{r} \ge \frac{\alpha -1 +p}{p} a \,+ \frac{\beta-1}{q}(1-a),\\
    &{\rm (ii)} \quad 
    \frac{1}{r}\ge \frac{a}{p} + \frac{1-a}{q} -a,\\
    &{\rm (iii)}\quad  \frac{1}{r} \le \frac{a}{p}  + \frac{1-a}{q}\,\,\,\text{if} \,\,\, \frac{\alpha -1+p}{p}= \frac{\gamma -1}{r}=\frac{\beta-1}{q},
\end{align*}
then there exists $C>0$ such that for any functions $f \in C_0^\infty [0, R)$ the following inequality holds.
\begin{align}\label{CCKN ineq n=1}
C \left( \int_0^R \frac{|f(x)|^r}{(\log \frac{eR}{x})^\gamma} \frac{dx}{x} \right)^{\frac{1}{r}} 
\le \left( \int_0^R \frac{x^p \,|f'(x)|^p}{(\log \frac{eR}{x})^\alpha} \frac{dx}{x}\right)^{\frac{a}{p}} \left( \int_0^R \frac{|f(x)|^q}{(\log \frac{eR}{x})^\beta} \frac{dx}{x} \right)^{\frac{1-a}{q}}
\end{align}
\end{theorem}

We show Theorem \ref{Thm CCKN n=1} by using Proposition \ref{Prop CKN ball} with $n=1$. If we consider the reflection, 
we get the following lemma from Proposition \ref{Prop CKN ball}. 

\begin{lemma}\label{lem CKN ball n=1}
Let $0< a \le 1, p, q \ge 1, r>0$ and $\alpha', \beta', \gamma' > -1$. If the following conditions hold:
\begin{align*}
    &{\rm (i)'}\quad 
    \frac{\gamma' +1}{r} \ge \frac{\alpha'+1 -p}{p} a \,+ \frac{\beta'+1}{q}(1-a),\\
    &{\rm (ii)'}\quad 
    \frac{1}{r}\ge \frac{a}{p} + \frac{1-a}{q} -a,\\
    &{\rm (iii)'} \quad 
    \frac{1}{r} \le \frac{a}{p}  + \frac{1-a}{q}
    \,\,\,\text{if} \,\,\, \frac{\alpha'+1 -p}{p}= \frac{\gamma'+1}{r}=\frac{\beta'+1}{q},
\end{align*}
then there exists $C>0$ such that for any functions $g \in C_0^\infty [0,R)$, the following inequality holds.
\begin{align}\label{CKN ineq on B n=1}
    C \left( \int_0^R y^{\gamma'} \, |g(y)|^r \,dy \right)^{\frac{1}{r}}
    \le \left( \int_0^R y^{\alpha'} \, |g'(y)|^p \,dy \right)^{\frac{a}{p}} \left( \int_0^R y^{\beta'} \, |g(y)|^q \,dy \right)^{\frac{1-a}{q}}
\end{align}
\end{lemma}

\begin{proof}[Proof of Theorem \ref{Thm CCKN n=1}]
Let $A>0$ satisfy $A(\alpha +p-1) +p >0$. 
Consider the following transformation.
\begin{align*}
    f(x)=g(y),\,\,\text{where}\,\, \log \frac{eR}{x}= \left( \frac{R}{y} \right)^A
\end{align*}
Then we see that
\begin{align*}
    \int_0^R \frac{x^p \,|f'(x)|^p}{(\log \frac{eR}{x})^\alpha} \frac{dx}{x}
    &=A^{1-p} R^{-A (\alpha -1+p)} \int_0^R y^{\alpha'} \, |g'(y)|^p \,dy,\\
    \int_0^R \frac{|f(x)|^q}{(\log \frac{eR}{x})^\beta} \frac{dx}{x} &= A R^{-A(\beta -1)}\int_0^R y^{\beta'} \, |g(y)|^q \,dy,\\
    \int_0^R \frac{|f(x)|^r}{(\log \frac{eR}{x})^\gamma} \frac{dx}{x} &= A R^{-A(\gamma -1)} \int_0^R y^{\gamma'} \, |g(y)|^r \,dy,
\end{align*}
where
\begin{align*}
    \alpha' &= A(\alpha -1+p) +p -1 >-1,\\
    \beta' &= A(\beta -1) -1 >-1,\\
    \gamma' &= A(\gamma -1) -1>-1.
\end{align*}
    Since $(a, p, q, r, \alpha, \beta, \gamma)$ satisfies (i)-(iii) in Theorem \ref{Thm CCKN n=1}, $(a, p, q, r,\alpha', \beta', \gamma')$ satisfies (i)'-(iii)' in Lemma \ref{lem CKN ball n=1}. Therefore, we can derive (\ref{CCKN ineq n=1}) from (\ref{CKN ineq on B n=1}).
\end{proof}

%
%
\section{Sufficiency part of Theorem \ref{Thm CCKN} when $a<1$ and $n \ge 2$}\label{S a<1}

%
%
\subsection{The case where $\frac{1}{r} \le \frac{a}{p}+ \frac{1-a}{q}$}\label{S a<1 n>1}

In this section, we consider the case where $a<1, n \ge 2, \frac{1}{r} \le \frac{a}{p}+ \frac{1-a}{q}$ and show the inequalities (\ref{CCKN ineq}) in Theorem \ref{Thm CCKN} by using assumptions (\ref{n-dilation valance})-(\ref{n-valance at 0}). The strategy of the proof is the same as it of Theorem \ref{Thm a=1 CCKN} in \S \ref{S a=1}. 

\begin{proof}[Proof of Theorem \ref{Thm CCKN}(Sufficiency when $\frac{1}{r} \le \frac{a}{p}+ \frac{1-a}{q}$)]
    We assume $a<1, n \ge 2, R=1$ and $\frac{1}{r} \le \frac{a}{p}+ \frac{1-a}{q}$.  

\noindent
{\bf [Step 1: Radial functions]} First, we will show the inequalities (\ref{CCKN ineq}) for any radial functions. Since the condition (\ref{n-valance at mid}) implies the condition (ii) in Theorem \ref{Thm CCKN n=1}, we can derive (\ref{CCKN ineq}) for radial functions from one-dimensional inequalities (\ref{CCKN ineq n=1}).   

\noindent
{\bf [Step 2: Spherical mean zero functions]} Let $A_k$ be the annulus given in the proof of Theorem \ref{Thm a=1 CCKN}. 
Similarly, we claim that the inequalities (\ref{CCKN ineq}) hold on the annulus $A_k$ for any spherical mean zero functions as follows. We will show Lemma \ref{lem:mean zero} later. 

\begin{lemma}\label{lem:mean zero}
Suppose the conditions (\ref{n-dilation valance}), (\ref{n-valance at mid}) in Theorem \ref{Thm CCKN}. Then 
    there exists $C>0$ such that for any $k \in \N$ and any $u \in C^\infty (B_1)$ with $\int_{\mathbb{S}^{n-1}} u(t\w) \,dS_\w =0$ for any $t \in [0, 1)$, 
    the following inequality holds.
    \begin{align}\label{CCKN A_k gene}
     \int_{A_k} \frac{|u(x)|^r }{|x|^n (\log \frac{e}{|x|})^\gamma} \,dx
    \le C  \left( \int_{A_k} \frac{|x|^p \,|\nabla u(x) |^p}{|x|^n (\log \frac{e}{|x|})^{\alpha}} \,dx \right)^{\frac{r}{p}a} \left( \int_{A_k} \frac{|u(x) |^q}{|x|^n (\log \frac{e}{|x|})^{\beta}} \,dx \right)^{\frac{r}{q}(1-a)}
\end{align}
\end{lemma}
Note that if $k=0$, then we also obtain the inequalities (\ref{CCKN A_k gene}) on $A_0$ for any spherical mean zero functions $u \in C^\infty (B_1)$ in the same way. 
If we get Lemma \ref{lem:mean zero}, then we get the inequalities (\ref{CCKN ineq}) on $B_1$ for any spherical mean zero functions $u$ as follows.
\begin{align*}
\int_{B_1} \frac{|u(x)|^r }{|x|^n (\log \frac{e}{|x|})^\gamma} \,dx
&= \sum_{k=0}^\infty
     \int_{A_k} \frac{|u(x)|^r }{|x|^n (\log \frac{e}{|x|})^\gamma} \,dx\\
    &\le C  \sum_{k=0}^\infty \left( \int_{A_k} \frac{|x|^p \,|\nabla u(x) |^p}{|x|^n (\log \frac{e}{|x|})^{\alpha}} \,dx \right)^{\frac{r}{p}a} \left( \int_{A_k} \frac{|u(x) |^q}{|x|^n (\log \frac{e}{|x|})^{\beta}} \,dx \right)^{\frac{r}{q}(1-a)}\\
    &\le C  \left( \int_{B_1} \frac{|x|^p \,|\nabla u(x) |^p}{|x|^n (\log \frac{e}{|x|})^{\alpha}} \,dx \right)^{\frac{r}{p}a} \left( \int_{B_1} \frac{|u(x) |^q}{|x|^n (\log \frac{e}{|x|})^{\beta}} \,dx \right)^{\frac{r}{q}(1-a)}
\end{align*}
In the last inequality, we used $\frac{1}{r} \le \frac{a}{p}+ \frac{1-a}{q}$ and the inequality:
\begin{align*}
    \sum_{k=0}^\infty |a_k|^c |b_k|^d 
    \le \left( \,\sum_{k=0}^\infty |a_k| \,\right)^c \left( \,\sum_{k=0}^\infty |b_k| \,\right)^d < \infty 
\end{align*}
for $c+d \ge 1, c, d \ge 0$ and $\{ a_k \}_{k=0}^\infty, \{ b_k \}_{k=0}^\infty \subset \re$. 
Therefore, we get the inequality (\ref{CCKN ineq}) for any spherical mean zero functions $u \in C^\infty (B_1)$.

\noindent
{\bf [Step 3: Any functions]} Finally, we will show (\ref{CCKN ineq}) for any functions $u \in C_0^\infty (B_1)$ by using [Step 1] and [Step 2]. 
For $u$, we consider its spherical average
\begin{align*}
    U(x)=U(t)=\frac{1}{\w_n} \int_{\mathbb{S}^{n-1}} u(t\w) \,dS_\w \quad (t=|x|)
\end{align*}
In the similar way to the proof of Theorem \ref{Thm a=1 CCKN}, we have
\begin{align*}
    \int_{B_1} \frac{|x|^p \,|\nabla U |^p}{|x|^n (\log \frac{e}{|x|})^{\alpha}} \,dx 
    &\le \int_{B_1} \frac{|x|^p \,|\nabla u |^p}{|x|^n (\log \frac{e}{|x|})^{\alpha}} \,dx,\\
    \int_{B_1} \frac{| U |^q}{|x|^n (\log \frac{e}{|x|})^{\beta}} \,dx 
    &\le \int_{B_1} \frac{| u |^q}{|x|^n (\log \frac{e}{|x|})^{\beta}} \,dx. 
\end{align*}
Since $u-U$ is a spherical mean zero function, we can apply [Step 2] for $u-U$. Thus, by using the inequality in $L^r$: 
\begin{align*}
    \| f+g \|_r \le C_r \, (\| f \|_r + \| g \|_r), \,\,\text{where} \,\, C_r = \begin{cases}
        1 \quad &(r \ge 1),\\
        2^{\frac{1}{r}-1} &(0< r < 1)
    \end{cases}
\end{align*}
and the triangle inequality in $L^p$ and $L^q$, 
we have
\begin{align*}
     &\left( \int_{B_1} \frac{|u|^r \,dx}{|x|^n (\log \frac{e}{|x|})^\gamma} \right)^{\frac{1}{r}} - C_r\left( \int_{B_1} \frac{|U|^r \,dx}{|x|^n (\log \frac{e}{|x|})^\gamma} \right)^{\frac{1}{r}}
     \le C_r \left( \int_{B_1} \frac{|u-U|^r \,dx}{|x|^n (\log \frac{e}{|x|})^\gamma} \right)^{\frac{1}{r}} \\
    &\le C  \left( \int_{B_1} \frac{|x|^p \,|\nabla u- \nabla U |^p\,dx}{|x|^n (\log \frac{e}{|x|})^{\alpha}}  \right)^{\frac{a}{p}}
    \left( \int_{B_1} \frac{| u-U |^q}{|x|^n (\log \frac{e}{|x|})^{\beta}} \,dx \right)^{\frac{1-a}{q}}\\
    &\le C \left( \int_{B_1} \frac{|x|^p \,|\nabla u |^p\,dx}{|x|^n (\log \frac{e}{|x|})^{\alpha}}  \right)^{\frac{a}{p}}
    \left( \int_{B_1} \frac{| u|^q}{|x|^n (\log \frac{e}{|x|})^{\beta}} \,dx \right)^{\frac{1-a}{q}}.
\end{align*}
Therefore, using [Step 1] for $U$, we have
\begin{align*}
    &\left( \int \frac{|u|^r}{|x|^n (\log \frac{e}{|x|})^\gamma} \right)^{\frac{1}{r}} 
    \le C \left( \int \frac{|x|^p \,|\nabla u |^p}{|x|^n (\log \frac{e}{|x|})^{\alpha}}  \right)^{\frac{a}{p}} \left( \int \frac{| u|^q}{|x|^n (\log \frac{e}{|x|})^{\beta}} \right)^{\frac{1-a}{q}}+ C_r \left( \int\frac{|U|^r }{|x|^n (\log \frac{e}{|x|})^\gamma} \right)^{\frac{1}{r}}\\
    &\le C \left\{ \left( \int \frac{|x|^p|\nabla u |^p}{|x|^n (\log \frac{e}{|x|})^{\alpha}}  \right)^{\frac{a}{p}} \left( \int \frac{| u|^q }{|x|^n (\log \frac{e}{|x|})^{\beta}} \right)^{\frac{1-a}{q}}
    +  \left( \int \frac{|x|^p |\nabla U |^p}{|x|^n (\log \frac{e}{|x|})^{\alpha}} \right)^{\frac{a}{p}} \left( \int \frac{| U|^q}{|x|^n (\log \frac{e}{|x|})^{\beta}} \right)^{\frac{1-a}{q}} \right\}\\
    &\le C \left( \int_{B_1} \frac{|x|^p \,|\nabla u |^p\,dx}{|x|^n (\log \frac{e}{|x|})^{\alpha}}  \right)^{\frac{a}{p}} \left( \int_{B_1} \frac{| u|^q}{|x|^n (\log \frac{e}{|x|})^{\beta}} \,dx \right)^{\frac{1-a}{q}}. 
\end{align*}
Therefore, we get the inequalities (\ref{CCKN ineq}). 
\end{proof}

To show Lemma \ref{lem:mean zero}, we just use the Gagliardo-Nirenberg inequality below instead of the Sobolev inequality.  

\begin{prop}\label{Prop GN}(The Gagliardo-Nirenberg inequality with mean zero)\\
  Let $n \ge 2, A$ be a domain in $\re^n$ with smooth bounded boundary and $p,q \ge 1, \tilde{r} >0, a\in(0,1)$ satisfy 
   \begin{align*}
    \frac{1}{\tilde{r}}=\frac{a}{p} - \frac{a}{n} +\frac{1-a}{q}. 
   \end{align*}
   Then there exists $C>0$ such that for any functions $u\in C^{\infty}(A)$ with $\int_A u(y)\,dy =0$, the following inequality holds.
   \begin{align}\label{GN gene}
       \left(\int_{A}|u(x)|^{\tilde{r}} \,dx\right)^{\frac{1}{\tilde{r}}}\le C\left(\int_{A}|\nabla u(x)|^p\,dx\right)^{\frac{a}{p}}\left(\int_{A}|u(x)|^{q}\,dx\right)^{\frac{1-a}{q}}
   \end{align}
\end{prop}

In \cite{CKN}, they used Proposition \ref{Prop GN} on annulus to show Theorem A by referring \cite{G, N}. However, we cannot find an exact proof of the Gagliardo-Nirenberg inequality with mean zero in \cite{G, N}. Thus, we will give the outline of a proof of Proposition \ref{Prop GN} in Appendix.


\begin{proof}[Proof of Lemma \ref{lem:mean zero}] 
Thanks to the nonlinear scaling argument used in the proof of Lemma \ref{lem:mean zero a=1}, it is enough to show the inequalities (\ref{CCKN A_k gene}) only in $A_1$. Furthermore, every logarithmic weight in (\ref{CCKN A_k gene}) is bounded in $A_1$. Thus, it is enough to show  
\begin{align}\label{GN on A_1}
    \left(\int_{A_1}|u(x)|^{r} \,dx\right)^{\frac{1}{r}}\le C\left(\int_{A_1}|\nabla u(x)|^p\,dx\right)^{\frac{a}{p}}\left(\int_{A_1}|u(x)|^{q}\,dx\right)^{\frac{1-a}{q}}
\end{align}
for spherical mean zero functions $u \in C^{\infty}(A_1)$. Note that spherical mean zero function in $A_1$ is especially mean zero function in $A_1$.  

\noindent
{\bf [(I)\,The case: $\frac{a}{p} -\frac{a}{n} + \frac{1-a}{q} >0$]} Let $\tilde{r} >0$ satisfy
$$\frac{1}{\tilde{r}}=\frac{a}{p} - \frac{a}{n} +\frac{1-a}{q}.$$
By using (\ref{n-valance at mid}), we see that $\tilde{r} \ge r$. Thus, by using the H\"older inequality and Proposition \ref{Prop GN}, we have 
\begin{align*}
    \left(\int_{A_1}|u|^{r} \,dx\right)^{\frac{1}{r}}\le C \left(\int_{A_1}|u|^{\tilde{r}} \,dx\right)^{\frac{1}{\tilde{r}}}\le C\left(\int_{A_1}|\nabla u(x)|^p\,dx\right)^{\frac{a}{p}}\left(\int_{A_1}|u|^{q}\,dx\right)^{\frac{1-a}{q}}
\end{align*}

\noindent
{\bf [(II)-(i)\,The case: $\frac{a}{p} -\frac{a}{n} + \frac{1-a}{q} \le 0$ and $q \ge r$]} Since $\frac{1-a}{q} >0$, we see $p>n$. Thus, by using the H\"older inequality, the Morrey inequality and the Poincar\'e inequality with mean zero, we have
\begin{align*}
    \left(\int_{A_1}|u|^{r} \,dx\right)^{\frac{1}{r}}
    &\le C \left(\int_{A_1}|u|^{q} \,dx\right)^{\frac{a}{q} + \frac{1-a}{q}}
    \le C \, \| u \|_{L^\infty (A_1)}^a \left(\int_{A_1}|u|^{q} \,dx\right)^{\frac{1-a}{q}}\\
    &\le C \| u \|_{W^{1,p} (A_1)}^a \left(\int_{A_1}|u|^{q} \,dx\right)^{\frac{1-a}{q}}
    \le C \left(\int_{A_1}|\nabla u|^{p} \,dx\right)^{\frac{a}{p}} \left(\int_{A_1}|u|^{q} \,dx\right)^{\frac{1-a}{q}}.
\end{align*}

\noindent
{\bf [(II)-(ii)\,The case: $\frac{a}{p} -\frac{a}{n} + \frac{1-a}{q} \le 0$ and $q<r$]} Note that $p>n$. Let $\tilde{a} >0$ satisfy
$$\tilde{a} = \frac{\frac{1}{q} - \frac{1}{r}}{\frac{1}{q} + \frac{1}{n} - \frac{1}{p}},\quad \text{that is},\quad  \frac{1}{r} = \frac{\tilde{a}}{p} -\frac{\tilde{a}}{n} + \frac{1-\tilde{a}}{q}.$$
By using (\ref{n-valance at mid}), we see that
\begin{align*}
    \frac{\tilde{a}}{p} -\frac{\tilde{a}}{n} + \frac{1-\tilde{a}}{q} \ge \frac{a}{p} -\frac{a}{n} + \frac{1-a}{q},\quad \text{that is},\quad 0 \ge (a-\tilde{a}) \left( \frac{1}{p}- \frac{1}{n} - \frac{1}{q} \right). 
\end{align*}
which implies $\tilde{a} \le a <1$. By using Proposition \ref{Prop GN}, the Morrey inequality and the Poincar\'e inequality with mean zero, we have
\begin{align*}
    \left(\int_{A_1}|u|^{r} \,dx\right)^{\frac{1}{r}} &\le C \left(\int_{A_1}|\nabla u|^{p} \,dx\right)^{\frac{\tilde{a}}{p}} \left(\int_{A_1}|u|^{q} \,dx\right)^{\frac{1-a}{q} + \frac{a-\tilde{a}}{q}}\\
    &\le C \left(\int_{A_1}|\nabla u|^{p} \,dx\right)^{\frac{a}{p}} \left(\int_{A_1}|u|^{q} \,dx\right)^{\frac{1-a}{q}}. 
\end{align*}
Therefore, we get (\ref{GN on A_1}). 
\end{proof}

\subsection{The case where $\frac{1}{r} > \frac{a}{p}+ \frac{1-a}{q}$}\label{S a<1 n>1 last}

In this section, we consider the case where $a<1, n \ge 2, \frac{1}{r} > \frac{a}{p}+ \frac{1-a}{q}$ and show the inequalities (\ref{CCKN ineq}) by using assumptions (\ref{n-dilation valance})-(\ref{n-valance at 0}), \S \ref{S a=1} and \S \ref{S a<1 n>1}.

\begin{proof}[Proof of Theorem \ref{Thm CCKN}(Sufficiency when $\frac{1}{r} > \frac{a}{p}+ \frac{1-a}{q}$)] 
Note that in \S \ref{S a<1 n>1}, we used the assumption $\frac{1}{r} \le \frac{a}{p} + \frac{1-a}{q}$ only when we show [Step 2]. Thus it is enough to show the inequalities (\ref{CCKN ineq}) for any spherical mean zero functions when $\frac{1}{r} > \frac{a}{p}+ \frac{1-a}{q}$. 
For spherical mean zero function $u \in C^\infty (B_1)$, we set
\begin{align*}
    A(u)= \left( \int_{B_1} \frac{|x|^p |\nabla u|^p}{|x|^n (\log \frac{e}{|x|})^\alpha}\,dx \right)^{\frac{1}{p}} >0\quad \text{and}\quad  B(u)= \left( \int_{B_1} \frac{|u|^q}{|x|^n (\log \frac{e}{|x|})^\beta}\,dx \right)^{\frac{1}{q}} >0.
\end{align*}
From (\ref{n-dilation valance}), (\ref{n-valance at 0}) and $\frac{1}{r} > \frac{a}{p}+ \frac{1-a}{q}$, we can also assume 
$$\frac{\alpha -1+p}{p} \not= \frac{\beta -1}{q}\quad {\rm or}\quad \frac{\gamma -1}{r} > \frac{\alpha -1+p}{p} = \frac{\beta -1}{q}.$$ 

\noindent
{\bf [(I)\,The case: $\frac{\alpha -1+p}{p} > \frac{\beta -1}{q}$]} 
For $b \in (0,a)$, we set $\delta$ and $s$ as follows.
\begin{align*}
    \frac{1}{s} = \frac{b}{p} + \frac{1-b}{q} \quad \text{and}\quad \frac{\delta -1}{s} = \frac{\alpha -1 + p}{p} b + \frac{\beta -1}{q} (1-b)
\end{align*}
Then we see that 
\begin{align}\label{s delta}
    0< \frac{\delta -1}{s} < \frac{\gamma -1}{r} \quad \text{and}\quad s>r \,\,\,\text{for} \,\,b\,\,\text{close to}\,\,a. 
\end{align}
We divide two cases as follows. 

\noindent
{\bf [(I)-(i)\,The case: $B(u) \le A(u)$]} By using the H\"older inequality, (\ref{s delta}), \S \ref{S a<1 n>1}, $b<a$ and $B(u) \le  A(u)$, we have
\begin{align*}
    \left( \int_{B_1} \frac{|u|^r \,dx}{|x|^n (\log \frac{e}{|x|})^\gamma} \right)^{\frac{1}{r}}
    &\le \left( \int_{B_1} \frac{|u|^s \,dx}{|x|^n (\log \frac{e}{|x|})^\delta} \right)^{\frac{1}{s}} \left( \int_{B_1} \frac{dx}{|x|^n (\log \frac{e}{|x|})^{\frac{\gamma s -r \delta}{s-r}}} \right)^{\frac{1}{r}-\frac{1}{s}}\\
    &\le C A(u)^b B(u)^{1-b} \le C A(u)^{a} B(u)^{1-a} 
\end{align*}

\noindent
{\bf [(I)-(ii)\,The case: $B(u)>  A(u)$]}
Consider the scaled function $v$ which is given by
\begin{align*}
    v(y) = \frac{u(x)}{D}, \,\text{where}\,\, y= \left( \frac{|x|}{e} \right)^{\frac{1}{\la}-1} x,\, \la^{\frac{\alpha -1+p}{p} -\frac{\beta -1}{q}} = \frac{B(u)}{A(u)} >1\,\, \text{and}\,\,D=\la^{\frac{\beta -1}{q}} B(u). 
\end{align*}
Note that $\la >1$ and $v \in C^\infty (B_{e^{1-\frac{1}{\la}}} \setminus \{ 0\}) \cap C(B_{e^{1-\frac{1}{\la}}})$ is also a spherical mean zero function. 
By similar calculations to them in \S \ref{S nonlinear}, we see that
\begin{align}\label{v and u}
    \int_{B_{e^{1-\frac{1}{\la}}}} \frac{|y|^p |\nabla v|^p\,dy}{|y|^n (\log \frac{e}{|y|})^\alpha} \le \frac{\la^{\alpha -1+p} A(u)^p}{D^p} =1,\int_{B_{e^{1-\frac{1}{\la}}}} \frac{|v|^q\,dy}{|y|^n (\log \frac{e}{|y|})^\beta} = \frac{\la^{\beta -1} B(u)^q}{D^q} =1.
\end{align}
By using the H\"older inequality, (\ref{s delta}), \S \ref{S a<1 n>1}, (\ref{v and u}) and (\ref{n-dilation valance}), we have
\begin{align}\label{lambda 1}
    \left( \int_{B_{e^{1-\lambda}}} \frac{|u|^r \,dx}{|x|^n (\log \frac{e}{|x|})^\gamma} \right)^{\frac{1}{r}}
    &= D \la^{-\frac{\gamma -1}{r}} \left( \int_{B_1} \frac{|v|^r \,dy}{|y|^n (\log \frac{e}{|y|})^\gamma} \right)^{\frac{1}{r}} \notag \\
    &\le C D\la^{-\frac{\gamma -1}{r}} \left( \int_{B_1} \frac{|v|^s \,dy}{|y|^n (\log \frac{e}{|y|})^\delta} \right)^{\frac{1}{s}} \notag \\
    &\le C D\la^{-\frac{\gamma -1}{r}}  A(v)^b B(v)^{1-b} \le C D\la^{-\frac{\gamma -1}{r}} \notag   \\
    &\le CD \la^{-\frac{\alpha -1+p}{p}a - \frac{\beta -1}{q}(1-a)} = C A(u)^a B(u)^{1-a}. 
\end{align}
On the other hand, for $d \in (a, 1)$, we set $\ep$ and $t$ as follows.
\begin{align*}
    \frac{1}{t} = \frac{d}{p} + \frac{1-d}{q} \quad \text{and}\quad \frac{\ep -1}{t} = \frac{\alpha -1 + p}{p} d + \frac{\beta -1}{q} (1-d)
\end{align*}
Temporary, we assume $\frac{\gamma -1}{r} = \frac{\alpha -1+p}{p}a + \frac{\beta -1}{q}(1-a)$. Then we see that 
\begin{align}\label{t ep}
     \frac{\ep -1}{t} > \frac{\gamma -1}{r} \quad \text{and}\quad t>r \,\,\,\text{for} \,\,d\,\,\text{close to}\,\,a.
\end{align}
By using the H\"older inequality, (\ref{t ep}), \S \ref{S a<1 n>1}, (\ref{v and u}) and (\ref{n-dilation valance}), we have
\begin{align}\label{lambda 2}
    &\left( \int_{B_1 \setminus B_{e^{1-\la}}} \frac{|u|^r \,dx}{|x|^n (\log \frac{e}{|x|})^\gamma} \right)^{\frac{1}{r}}
    = D \la^{-\frac{\gamma -1}{r}} \left( \int_{B_{e^{1-\frac{1}{\la}}} \setminus B_1} \frac{|v|^r \,dy}{|y|^n (\log \frac{e}{|y|})^\gamma} \right)^{\frac{1}{r}} \notag \\
    &\le  D \la^{-\frac{\gamma -1}{r}} \left( \int_{B_{e^{1-\frac{1}{\la}}} \setminus B_1} \frac{|v|^t \,dy}{|y|^n (\log \frac{e}{|y|})^\ep} \right)^{\frac{1}{t}} 
    \left( \int_{B_e \setminus B_1} \frac{dy}{|y|^n (\log \frac{e}{|y|})^{\frac{\gamma t -r \ep}{t-r}}} \right)^{\frac{1}{r}-\frac{1}{t}} \notag \\ 
    &\le C \la^{\frac{\ep -1}{t} - \frac{\gamma -1}{r}} \left( \int_{B_1 \setminus B_{e^{1-\la}}} \frac{|u|^t \,dx}{|x|^n (\log \frac{e}{|x|})^\ep} \right)^{\frac{1}{t}} \notag \\
    &\le C \la^{(\frac{\alpha -1+p}{p} -\frac{\beta -1}{q})(d-a)} 
    A(u)^d B(u)^{1-d} = C A(u)^a B(u)^{1-a}. 
    \end{align}
From (\ref{lambda 1}) and (\ref{lambda 2}), we have the inequalities (\ref{CCKN ineq}). If $\frac{\gamma -1}{r} > \frac{\alpha -1+p}{p}a + \frac{\beta -1}{q}(1-a)$, then we set $\tilde{\gamma} \in (1, \gamma)$ as $\frac{\tilde{\gamma} -1}{r} = \frac{\alpha -1+p}{p}a + \frac{\beta -1}{q}(1-a)>0$. Since $(\log \frac{e}{|x|})^\gamma \ge (\log \frac{e}{|x|})^{\tilde{\gamma}}$ for any $x \in B_1 \setminus \{ 0\}$, we can reduce it to the case $\frac{\gamma -1}{r} = \frac{\alpha -1+p}{p}a + \frac{\beta -1}{q}(1-a)$.

\noindent
{\bf [(II)\,The case: $\frac{\alpha -1+p}{p} < \frac{\beta -1}{q}$]} 
This case can be shown in a similar way to the previous case if we set $\tilde{\beta} >\beta$ as $\frac{\gamma -1}{r}=\frac{\alpha -1+p}{p}a + \frac{\tilde{\beta} -1}{q}(1-a)$ instead of setting $\tilde{\gamma}$ in the previous case. Therefore, we omit the proof.

\noindent
{\bf [(III)\,The case: $\frac{\gamma -1}{r} > \frac{\alpha -1+p}{p} = \frac{\beta -1}{q}$]} In this case, we set $\tilde{r}$ and $\tilde{\gamma}$ as follows.
\begin{align*}
    \frac{1}{\tilde{r}} = \frac{a}{p} + \frac{1-a}{q} \left(  < \frac{1}{r} \right) \quad \text{and}\quad  \frac{\tilde{\gamma} -1}{\tilde{r}} = \frac{\alpha -1 + p}{p} = \frac{\beta -1}{q} 
\end{align*}
Then we see that 
\begin{align*}
    0< \frac{\tilde{\gamma} -1}{\tilde{r}} < \frac{\gamma -1}{r} \quad \text{and}\quad \tilde{r} >r.
\end{align*}
By using the H\"older inequality and \S \ref{S a<1 n>1}, we have
\begin{align*}
    \left( \int_{B_1} \frac{|u|^r \,dx}{|x|^n (\log \frac{e}{|x|})^\gamma} \right)^{\frac{1}{r}}
    &\le C \left( \int_{B_1} \frac{|u|^{\tilde{r}} \,dx}{|x|^n (\log \frac{e}{|x|})^{\tilde{\gamma}}} \right)^{\frac{1}{\tilde{r}}}
    \le C A(u)^{a} B(u)^{1-a} 
\end{align*}
Therefore, we complete the proof of Theorem \ref{Thm CCKN}. 
\end{proof}

\begin{remark}\label{Rem scale}
    For original CKN inequalities (\ref{CKN ineq}), it is possible to normalize the norms $A(u)$ and $B(u)$ by using scale invariance with respect to dilation. However, for critical CKN inequalities (\ref{CCKN ineq}), it seems difficult to normalize due to the lack of scale invariance. Therefore, we divide two cases (I)-(i) and (I)-(ii) as above. 
\end{remark}

%
%
\section{Appendix: Proof of Proposition \ref{Prop GN}}\label{S App}

In this section, we give the outline of a proof of the Gagliardo-Nirenberg inequality with mean zero on a domain $A$ with smooth bounded boundary.  

\begin{proof}[Proof of Proposition \ref{Prop GN}]
Thanks to the Poincar\'e inequality with mean zero, it is enough to show 
\begin{align}\label{weak GN}
    \| u \|_{L^{\tilde{r}} (A)} \le C \,\| u \|_{W^{1,p} (A) }^a \,\| u \|_{L^q (A)}^{1-a}
\end{align} 
which is a weaker version of (\ref{GN gene}), for any functions $u \in C^{\infty} (A)$. 
By using the Gagliardo-Nirenberg inequality (\ref{weak GN}) on $\re^n$, see e.g. \cite[p.400, Theorem 12.83]{Leoni}, \cite{OT}, and an extension operator $E$ below, we shall show (\ref{weak GN}).    

\begin{prop}\label{Prop extension}
Let $A \subset \re^n$ be a domain and $\partial A$ be $C^1$ and bounded. Then there exists a linear operator $E: L^1_{\rm loc}(\overline{A}) \to L^1_{\rm loc}(\mathbb{R}^n)$ such that the followings hold:
\begin{itemize}
    \item[\rm (i)] For any $u \in L^{1}_{\rm loc} (\overline{A})$, $Eu=u$ in $A$.
    \item[\rm (ii)] For any $p, q \ge 1$, $E(L^q(A)) \subset L^q(\re^n)$ and $E(W^{1,p}(A)) \subset W^{1,p}(\re^n)$.
    \item[\rm (iii)] For any $p, q \ge 1$, there exist $C_1=C_1(p)>0$ and $C_2 = C_2(q)>0$ such that 
\begin{align*}
\| Eu \|_{W^{1,p}(\re^n)} \le C_1 \,\| u \|_{W^{1,p}(A)}\quad &(\forall u \in W^{1,p}(A)),\\
    \| Eu \|_{L^q(\re^n)} \le C_2 \,\| u \|_{L^q(A)}\quad &(\forall u \in L^q(A)). 
\end{align*}
\end{itemize}
\end{prop}

\noindent
Let $E$ be the extension operator as above. By using GN inequality on $\re^n$, we have
\begin{align*}
\| u \|_{L^{\tilde{r}}(A)}
&\le \| Eu \|_{L^{\tilde{r}}(\mathbb{R}^n)} 
\le C \,\| Eu \|_{W^{1,p}(\mathbb{R}^n)}^a \| Eu \|_{L^q(\mathbb{R}^n)}^{1-a} 
\le C \,\|u \|_{W^{1,p}(A)}^a \|u \|_{L^q(A)}^{1-a} 
\end{align*} 
Therefore, we get (\ref{weak GN}).
\end{proof}

\begin{remark}
    Concerning to Proposition \ref{Prop extension}, it is often that the domain of extension operator is Sobolev space. However, if we review a proof of extension operator carefully, we see that the domain of extension operator $E$ can be $L^1_{\rm loc}(A)$ and for the same $E$, we get two estimates in Proposition \ref{Prop extension} {\rm (iii)}, see e.g. \cite[Section VI.3 and VI.4]{Stein}, \cite[Chapter 13]{Leoni}, \cite[Proposition 2.3]{BM}. 
\end{remark}


\section*{Acknowledgment}
The authors would like to thank Prof. N. Yessirkegenov (KIMEP University) for his constructive comments.
The first author (M.S.) was supported by JSPS KAKENHI Early-Career Scientists, No.23K13001.
The second author (Y.S.) was supported by Science Committee of the Ministry of Science and Higher Education of the Republic of Kazakhstan (Grant No. AP23490970). 
Part of this work was conducted while the first and third authors were visiting SDU University in 2025. The authors thank SDU University for their hospitality. 
Also, this work was partly supported by Osaka Central University Advanced Mathematical Institute (MEXT Joint Usage/Research Center on Mathematics and Theoretical Physics). 



\begin{thebibliography}{99}


\bibitem{AHN}
Ando, H., Horiuchi, T., Nakai, E., 
{\it On the critical Caffarelli-Kohn-Nirenberg type inequalities involving super-logarithms}, 
Hiroshima Math. J. 55 (2025), no. 1, 65-87.


\bibitem{Adams}
Adams, D. R., {\it Weighted nonlinear potential theory}, 
\newblock Trans. Amer. Math. Soc. \textbf{297} (1986), no. 1, 73--94.


\bibitem{AS}
Adimurthi, Sandeep, K., {\it Existence and non-existence of the first eigenvalue of the perturbed Hardy-Sobolev operator}, 
\newblock Proc. Roy. Soc. Edinburgh Sect. A \textbf{132} (2002), No.5, 1021--1043.





\bibitem{BG}
Baras, P., Goldstein, J., 
{\it The heat equation with a singular potential}, 
Trans. Amer. Math. Soc. 284 (1984), no. 1, 121-139.


\bibitem{BM}
Brezis, H., Mironescu, P., {\it Gagliardo-Nirenberg inequalities and non-inequalities: the full story}, Ann. Inst. H. Poincar\'e C Anal. Non Lin\'eaire 35 (2018), no. 5, 1355-1376.

\bibitem{BV}
Brezis, H., V\'{a}zquez, J. L., {\it Blow-up solutions of some nonlinear elliptic problems}, 
\newblock Rev. Mat. Univ. Complut. Madrid 10 (1997), No. 2, 443-469.

\bibitem{CKN}
Caffarelli, L. A., Kohn, R., Nirenberg, L., {\it First order interpolation inequalities with weights}, 
\newblock Composito Math. \textbf{53} (1984), 259--275.

\bibitem{CKN(PDE)}
Caffarelli, L. A., Kohn, R., Nirenberg, L., {\it Partial regularity of suitable weak solutions of the Navier-Stokes equations}, 
\newblock Comm. Pure Appl. Math, \textbf{35} (1982), 771--831.

\bibitem{Evans}
Evans, L. C., {\it Partial differential equations. Second edition}, Graduate Studies in Mathematics, 19. American Mathematical Society, Providence, RI, 2010. xxii+749 pp.

\bibitem{G}
Gagliardo, E., {\it Ulteriori propritet\`a di alcune classi di funzioni in pi\`u variabili},
Ricerche di Mat. Napoli 8 (1959) 24-51.



\bibitem{H}
Horiuchi, T., {\it On general Caffarelli-Kohn-Nirenberg type inequalities involving non-doubling weights}, Sci. Math. Jpn. (in Editions Electronica), e-2022-10, 16 pages. 

\bibitem{H p=1}
Horiuchi, T., {\it On general Caffarelli-Kohn-Nirenberg type inequalities involving non-doubling weights in the case of $p=1$}, arXiv:2512.21492. 

\bibitem{HK}
Horiuchi, T., Kumlin, P., {\it On the Caffarelli-Kohn-Nirenberg-type inequalities involving critical and supercritical weights}, 
\newblock Kyoto J. Math. \textbf{52} (2012), no. 4, 661--742. 

\bibitem{II}
Ioku, N., Ishiwata, M., {\it A Scale Invariant Form of a Critical Hardy Inequality}, 
\newblock Int. Math. Res. Not. IMRN (2015), no. 18, 8830--8846.

\bibitem{Leoni}
Leoni, G. {\it A first course in Sobolev spaces. Second edition.}, Grad. Stud. Math., 181 American Mathematical Society, Providence, RI, 2017. xxii+734 pp.

\bibitem{Leray}
Leray, J., {\it Etude de diverses equations integrales non lineaires et de quelques problemes que pose l'hydrodynamique. (French)}, (1933), 82 pp.

\bibitem{MOW(Tohoku)}
Machihara, S., Ozawa, T., Wadade, H., {\it Hardy type inequalities on balls}, Tohoku Math. J. (2) 65 (2013), no. 3, 321-330.


\bibitem{Nash}
Nash, J., {\it Continuity of solutions of parabolic and elliptic equations}, Amer. J.Math. 80 (1958),
931-954.

\bibitem{N}
Nirenberg, L., 
{\it On elliptic partial differential equations}, 
Ann. Scuola Norm. Sup. Pisa Cl. Sci. (3) 13 (1959), 115-162.

\bibitem{OT}
Ozawa, T. Takeuchi, T., {\it A new proof of the Gagliardo-Nirenberg and Sobolev inequalities: heat semigroup approach}, Proc. Amer. Math. Soc. Ser. B 11 (2024), 371-377.

\bibitem{P}
Psaradakis, G., {\it Limiting case of an isoperimetric inequality with radial densities and applications}, 
Comm. Anal. Geom. 30 (2022), no. 6, 1391-1411.


\bibitem{RSY2}
Ruzhansky, M., Suragan, D., Yessirkegenov, N., 
{\it Extended Caffarelli-Kohn-Nirenberg inequalities, and remainders, stability, and superweights for $L^p$-weighted Hardy inequalities}, 
Trans. Amer. Math. Soc. Ser. B 5 (2018), 32-62.

\bibitem{RSY}
Ruzhansky, M., Suragan, D., Yessirkegenov, N., {\it Extended Caffarelli-Kohn-Nirenberg inequalities and superweights for $L^p$-weighted Hardy inequalities}, 
C. R. Math. Acad. Sci. Paris 355 (2017), no. 6, 694-698.

\bibitem{S(JDE)}
Sano, M., {\it Extremal functions of generalized critical Hardy inequalities}, 
\newblock J. Differential Equations \textbf{267} (2019), no. 4, 2594--2615.

\bibitem{S(MJM)}
Sano, M., {\it Improvements and generalizations of two Hardy type inequalities and their applications to the Rellich type inequalities}, Milan J. Math. 90 (2022), no. 2, 647-678.


\bibitem{Sano-TF(JGA)}
Sano, M., Takahashi, F., 
{\em On eigenvalue problems involving the critical Hardy potential and Sobolev type inequalities with logarithmic weights in two dimensions},
\newblock J. Geometric Anal., {\bf 34}, no.4, Paper No. 112 (32 pages), (2024) 


\bibitem{ST(HT)}
Sano, M., Takahashi, F., {\it The critical Hardy inequality on the half-space via harmonic transplantation}, 
\newblock Calc. Var. Partial Differential Equations \textbf{61} (2022), no. 4, Paper No. 158, 33 pp.

\bibitem{SSW}
Smets, D., Su, J., Willem, M., {\it Non-radial ground states for the H\'enon equation}, Commun. Contemp. Math. 4 (2002), no. 3, 467-480.


\bibitem{Stein}
Stein, E.M., {\it Singular Integrals and Differentiability Properties of Functions}, Princeton Math. Ser., vol. 30, Princeton University Press, Princeton, NJ, 1970.


\bibitem{Struwe}
Struwe, M., {\it Variational methods. Applications to nonlinear partial differential equations and Hamiltonian systems. Fourth edition}, Ergebnisse der Mathematik und ihrer Grenzgebiete. 3. Folge. A Series of Modern Surveys in Mathematics [Results in Mathematics and Related Areas. 3rd Series. A Series of Modern Surveys in Mathematics], 34. Springer-Verlag, Berlin, (2008).



\end{thebibliography}
\end{document}